\documentclass[12pt]{amsart}
\usepackage[margin=0.8in]{geometry}
\newcommand{\norm}[1]{\lVert#1\rVert}
\usepackage[utf8]{inputenc}
\usepackage[english]{babel}
\usepackage{color}
\usepackage{soul,xcolor}
\usepackage{cancel}

\usepackage{amsmath, amssymb, amsfonts}
\DeclareMathOperator{\Exp}{Exp}
\DeclareMathOperator{\TV}{TV}
\DeclareMathOperator{\tr}{tr}

\numberwithin{equation}{section}

\newtheorem{theorem}{Theorem}[section]

\newlength{\defbaselineskip}
\newcommand{\setlinespacing}[1]%
           {\setlength{\baselineskip}{#1 \defbaselineskip}}

\newcommand{\RR}{{\mathbb R}}
\newcommand{\ZZ}{{\mathbb Z}}
\newcommand{\NN}{{\mathbb N}}

\newcommand{\ra}{\rightarrow}

\newcommand{\beql}[1]{\begin{equation}\label{#1}}
\newcommand{\eeq}{\end{equation}}

\newcommand{\beqal}[1]{\begin{eqnarray}\label{#1}}
\newcommand{\eeqa}{\end{eqnarray}}
\newcommand{\beq}{\begin{displaymath}}
\newcommand{\eeqno}{\end{displaymath}}
\newcommand{\bali}[1]{\begin{align}\label{#1}}
\newcommand{\eali}{\begin{align}}
\newcommand{\balino}{\begin{align*}}
\newcommand{\ealino}{\begin{align*}}

\newcommand{\la}{\lambda}

\newcommand{\qandq}{\quad\mbox{and}\quad}

\newcommand{\qforq}{\quad\mbox{for}\quad}

\newcommand{\non}{\nonumber}

\newcommand{\baa}{\begin{eqnarray*}}
\newcommand{\eaa}{\end{eqnarray*}}

\def\ups{\upsilon}

\newcommand{\ttl}{\large
Stationary Distributions and Convergence for $M/M/1$ Queues   \\[3pt] 
  in Interactive Random Environment \\[3pt]}

\begin{document}

\theoremstyle{plain}
\newtheorem{thm}{Theorem}[section]
\newtheorem{lemma}[thm]{Lemma}
\newtheorem{prop}[thm]{Proposition}
\newtheorem{cor}[thm]{Corollary}

\theoremstyle{definition}
\newtheorem{defn}{Definition}
\newtheorem{asmp}{Assumption}[section]
\newtheorem{rmk}{Remark}[section]
\newtheorem{exm}{Example}[section]

\author{Guodong Pang}

\address{Harold and Inge Marcus Department of Industrial and Manufacturing Engineering, Pennsylvania State University, University Park, PA 16802}

\email{gup3@psu.edu}

\author{Andrey Sarantsev}

\address{Department of Mathematics and Statistics, University of Nevada, Reno, NV 89557}

\email{asarantsev@unr.edu}

\author{Yana Belopolskaya}

\address{Saint Petersburg State University of Architecture and Civil Engineering, and  Petersburg Department of Mathematical Institute of Russian Academy of Science}

\email{yana@yb1569.spb.edu}

\author{Yuri Suhov}

\address{Statistical Laboratory, University of Cambridge; Department of Mathematics, Pennsylvania State University, University Park, PA 16802}

\email{yms@statslab.cam.ac.uk; ims14@psu.edu}

\title[Queues in Interactive Random Environment]{\ttl}

\begin{abstract}
A Markovian single-server queue is studied in an interactive random environment. 
The arrival and service rates of the queue depend on the environment, while the transition dynamics of the random environment depends on the queue length. We consider in detail two types of Markov random environments: a pure jump process and a reflected jump-diffusion. In both cases, the joint dynamics is constructed so that the stationary distribution can be explicitly found in a simple form (weighted geometric). 
We also derive an explicit estimate for exponential rate of convergence to the stationary distribution via coupling.
\end{abstract}

\maketitle

\thispagestyle{empty}

\section{Introduction}

In this paper we propose a tractable modeling approach to studying queues in an interactive random environment, where the arrival and/or service rates are modulated by a Markov process and the dynamics of the environment also depends on the state of the queue. Such models may be used in the following setting: In a service system (for example, on-demand service platforms), the demand may be affected by the service quality as indicated by dynamic `ratings' which may be modeled by a Markov chain, while the ratings dynamics may depend on the congestion level in the system. 

For an $M/M/1$ queue in an interactive random environment, let $N(t)$ be the queue-length process (the number of customers in the system) and $Z(t)$ be the random process for the environment. 
The joint process $(N(t), Z(t))$ can be modeled as a continuous-time Markov process on $\NN\times\mathcal{Z}$ ($\mathcal{Z}$ representing the range of $Z(t)$), with a generator
\begin{equation}\label{L-gen-general}
\mathcal{L}f(n, z)  = \mathcal{M}_zf(n, z) + \mathcal{A}_n f(n, z),
\end{equation}
where $\mathcal{M}_z$ describes the queueing dynamics depending on the environment state $z$ and 
$\mathcal{A}_n$ describes the environment dynamics depending on the queueing state $n$. 
Specifically, given the arrival and service rates $\lambda(z)$ and $\mu(z)$, we can write
\begin{equation*}
\mathcal M_zf(n,z) = \lambda(z)(f(n+1,z) - f(n,z))  + 1_{\{n \ne 0\}}\mu(z)(f(n-1,z) - f(n,z)).
\end{equation*}
On the other hand, the generator $\mathcal{A}_n$ can be for a general Markov process, depending on the queue length $n$. For example, for a given $n$, $\mathcal{A}_n$ may represent the generator of a diffusion process, 
$$
\mathcal{A}_n f(n,z) =  b_n(z)\cdot\nabla_z f(n,z) + \frac{1}{2}\tr\big(\Sigma_n(z)\nabla^2_z f(n,z)\big)
$$
or a  continuous-time jump Markov chain with a transition rate matrix depending on the queueing state $n$.  In the utmost generality, one can impose mild conditions on the generators
$\mathcal{M}_z$ and $\mathcal{A}_n$ to guarantee the existence of an invariant measure for the joint process $(N(t), Z(t))$. 
However, it seems difficult to go beyond that without any structural assumptions on the joint generator, especially $\mathcal{A}_n$.  In many applications, it is convenient to have an explicit invariant measure to work with. 
In general, it is hard to find an explicit form for stationary distributions of multidimensional Markov processes. (For example, in \cite{Williams} it is shown that an obliquely reflected Brownian motion (RBM) in a polyhedral domain  in $\RR^d$ has a product-of-exponentials stationary distribution under the {\it skew symmetry} condition, the only case with an explicit stationary measure.)

Therefore, in order to provide an \emph{explicit} expression for the invariant measure of the joint process, we  study a particular \emph{multiplicative} (scaled) form in the generator component $\mathcal{A}_n$, that is, 
$$
\mathcal{A}_nf(n,z) = \beta_n \rho^{-n}(z) \mathcal{A}f(n,z) 
$$
Here $\beta_n$ is a positive constant, $\rho(z) = \lambda(z)/\mu(z)$ is the traffic intensity in the queue, and $\mathcal{A}f(n,z)$ is a generator corresponding to  a Markov process whose transition dynamics does not depend on $n$. 
(In the case of reflected processes, the boundary conditions should be treated carefully; see Section \ref{sec-diffusive} for details.)
 The scaling factors not only depend on the queue length $n$, but also include the traffic intensity $\rho(z)$. 
 For an environment state $z$, $\rho^{-n}(z) >1$ for all queue state $n$, but the factor $\beta_n$ gives more flexibility (slowing down or speeding up) to the scaling of the generator $\mathcal{A}$. 
Our approach is motivated by applications where the environment dynamics may be sped up or slowed down by the congestion.
For example, in on-demand service systems, the transitions among the different  service quality `ratings' may simultaneously change faster when many customers experience more congestion due to higher response rates. 

We discuss two types of random environment: a pure jump Markov chain taking values in a discrete state space $D$ (finite or countable), and a reflected (jump) diffusion in a piecewise smooth domain, also denoted by $D$.  
Each type of environment is of its own interest. 
Under certain assumptions, we prove the existence of the joint invariant measure, derive its explicit expression and establish the exponential rate of convergence to the steady state (in the total variation norm). 
The explicit expression of the invariant measure can be regarded as a weighted geometric form (or  
some ``product form'', although not exactly in the same sense as in the literature on stochastic networks \cite{Kelly-book,Textbook3,KY-book}). Specifically, we have the joint invariant measure for $(N(t), Z(t))$ of the form $\pi(\{n\}, \mathrm{d}z) = \Xi^{-1}\rho^n(z)\nu(\mathrm{d}z),
$ where $\Xi$ is some normalization constant, and $\nu(\cdot)$ is the invariant measure associated with the generator $\mathcal{A}$. Recall that the steady-state distribution of the $M/M/1$ queue itself given an environment state $z$ is geometric ($P(N(\infty)=n) = (1-\rho(z))\rho^n(z)$). The product of the terms ``$\rho^n(z)$'' and ``$\nu(\mathrm{d}z)$" mixes the invariant measures for the queue and the environment, despite $\rho(z)$ depending on $z$. Here the scaling factor $\rho^{-n}(z)$ in $\mathcal{A}_n$  is critical.  
For the two types of environment processes we are able to establish the exponential rate of convergence.

With a diffusive environment, our work introduces new stochastic models. 
The simple models include:
(a) an $M/M/1$ queue with an interactive diffusive arrival rate: the arrival rate is a one-dimensional reflected (jump) diffusion process in $[0,1]$ under a fixed service rate $1$;
(b) an $M/M/1$ queue with an interactive diffusive service rate: the service rate is a one-dimensional reflected (jump) diffusion process in $[1,\infty)$ under a fixed arrival rate $1$;  
(c) the arrival and service rates form a two-dimensional RBM in an open convex cone (with arrival rate strictly lower than service rate).  
RBMs have been extensively studied in the queueing (network) literature as scaling limits. However, RBMs as arrival and/or service rates have not been carefully studied. 
When there is no interactive behavior, the $M/M/1$ queue with a RBM arrival rate can be regarded as  a special case of queues of the so-called doubly stochastic Poisson arrival processes with the arrival rate being an independent stochastic process (see, e.g., \cite{AP, BHZ05, BHZ06}). Our first model extends such existing interesting studies to include feedback loop from queue to environment.  The second and third models with RBM being the service rate or the RBM in the wedge for both arrival and service rates are new, even in the setting of no interactive behavior. Such models are worth further careful investigation. 
Of course, our models go beyond RBMs, to general reflected (jump) diffusion models.

We aim to find the explicit rate of convergence to the stationary distribution in these models. 
For standard $M/M/1$ queues, it is well known that the rate of convergence is exponential,  see, e.g., \cite[Proposition 5.8]{Robert-book}. 
However, for diffusion processes (solutions of SDEs), reflected diffusions, and their versions with jumps, the characterization of an explicit rate of convergence to steady state (as opposed to simply proving that there exists an exponential rate of convergence) is quite a challenging problem. See, e.g., \cite{CL89, MyOwn10, MyOwn12,MyOwn16}. 
Thus, it is a much more difficult problem to study the rate of convergence for the joint Markov process with a generator in the general form in \eqref{L-gen-general} due to the complicated interactive behavior of the two processes (one being discrete and the other being continuous). 
 We attempt to solve this problem via a {\it coupling} technique for the joint process $(N,Z)$. 
 We provide a novel way to construct the coupling time for the joint process in order to prove that the convergence rate is exponential,  and more importantly, provide good estimates of the rate of convergence via careful studies of the exponential bounds for the coupling time. 
 This appears to be the first work in the literature to carefully find the estimates of the coupling times of joint processes   for queueing processes in random environments.

Although our main focus is on the multiplicative (scaling) form in the generator $\mathcal{A}_n$, we have also considered a setup where the the environment jump diffusion described above depend on the queue length $n$ via its domain $D_n \subseteq D$. In particular, the drift vector field, covariance matrix field, and the jump measure remain the same for all $n$, but reflection vector fields may depend on the queueing state $n$.  The entire domain $D$ is the union of these $D_n,\, n = 0, 1, 2, \ldots$. We assume that this reflected jump-diffusion in $D_n$ has a unique invariant probability measure $\nu^{(n)}_{D_n}$ inside the domain $D_n$, which is the projection of a certain finite measure on $D$ to $D_n$. (The corresponding boundary measures may depend on $n$.)  See Assumptions \ref{asmp:domains}--\ref{asmp:finite-new-version}. We prove similar results as above in this setting. We construct two special examples:  an $M/M/1$ queue with a fixed service rate and a reflected diffusion arrival rate, controlled based on a threshold of queue length (Example \ref{ex:2.1}) and an $M/M/1$ queue with a fixed arrival rate and a diffusive service rate, controlled similarly (Example \ref{ex:2.2}). 

When the random environment is a Markov chain taking discrete values, our results also extend to the generator $\mathcal{A}_n$ of the form $\rho^{-n}(z)\tau_n(z,z')$, where the generator rate $\tau_n$ may depend on the queueing state $n$ unlike the multiplicative case. However, it is assumed that an invariant measure associated with the transition rate $\tau_n(z,z')$ exists such that it is independent of the queueing state $n$ (Assumption \ref{as-MM1-T}). This is slightly more general than the multiplicative case, so we state the model and results in section \ref{sec-jump} in this setup. We also give an example to illustrate how this slightly more general setup is used (see Examples \ref{ex-discrete-1} and  \ref{ex-discrete-2}). 

\subsection{Literature review on queues in interactive environments}

Queues in random environments (e.g., Markov modulated models) have been extensively studied in the literature. 
Most of the literature assumes that the queueing dynamics is affected by the environment, but not interactive. 
For example, the paper \cite{RS} studies Markov-modulated arrival and service rates with finite environment space, and finds expressions of waiting times. The paper \cite{Takine} deals with similar questions by comparing this queue with an appropriate $M/M/1$ queue. Optimization of service rate for the case when arrival rate is a Markov process is studied in \cite{KLT}. See also, a birth--death process in random environment \cite{CT81} and a Markov chain in Markov environment, studied in \cite{Cogburn, Economou, Matrix}. A particular case of a Markov--modulated setting is when the service dynamics is subject to interruptions. In this case, the random environment only affects service rate $\mu$. The survey \cite{KPC} summarizes the existing literature on this topic. 

In the Markov-modulated queueing literature, the arrival or service rates under modulation take finite or countable number of values. However, in practice, the rates under modulation can possibly take continuous values. Our work thus goes beyond the existing frameworks and develops new queueing models.

 In \cite{GPSY}, the authors study a random particle (a distinguished customer) walking randomly over the sites of a symmetric Jackson network (open or closed), where the arrival rate of a station/node or the transition of customers from it to other stations/nodes is affected if the particle occupies it, while the jump rate of the particle depends on the state of the station/node it currently occupies. An explicit steady state distribution for the joint process is derived. 
In \cite{KDO}, Jackson networks in interactive random environments are studied, where the service capacities are affected by the environment, while customer departing may enforce the environment to jump immediately. An explicit expression of the product form is derived for the joint queueing and environment processes.  Inspired by \cite{GPSY}, a different construction of Markov processes in random environments resulting in product-form invariant measure is provided. In \cite{BS}, various Markov processes with interactive random environment are constructed. This paper is in the same flavor as that in \cite{BS}. 
The paper \cite{Falin} deals with feedback loop created by blocking some channels in a multi-server queue, and finds a product-form stationary distribution for the joint process. 
None of these papers investigate the rate of convergence to stationarity. Our model of a single-server queue is also constructed in a more general manner.

The papers \cite{Cornez, Yechiali} study birth-death processes in random environment with feedback. This is a more general setup than in our paper, because an $M/M/1$ queue is a particular case of a birth-death process. However, \cite{Cornez} is concerned with explosion questions, rather than stationary distributions and convergence rates, and \cite{Yechiali} focuses on generating function approach and achieves only partial results for the steady-state distribution.

\subsection{Notation} The integral with respect to the measure $\nu$ applied to the function $f$ is written as $\langle\nu, f\rangle$.  Exponential distribution with rate $\alpha$ is denoted by $\Exp(\alpha)$. The arrow $\Rightarrow$ indicates weak convergence. The dot product of two vectors $a$ and $b$ is denoted by $a\cdot b$. We say two finite measures $\mu, \nu$ on $\mathbb R$ satisfy $\mu \preceq \nu$ if for all $u \in \mathbb R$ we have $\mu(-\infty, u] \le \nu(-\infty, u]$, but $\mu(\mathbb R) = \nu(\mathbb R)$. We say that $\mu$ is {\it stochastically dominated by} $\nu$. We transfer this concept to random variables: $X$ is stochastically dominated by $Y$ if the distribution of $X$ is stochastically dominated by the distribution of $Y$. Let $\ZZ_+ = \{0, 1, 2, \ldots\}$ and $\RR_+ := [0, \infty)$. Define the {\it total variation norm:} For a signed measure $\mathbf{\nu}$, let $\norm{\mathbf{\nu}}_{\TV}:= \sup_{A}|\mathbf{\nu}(A)|$.  Throughout this article, we consider continuous-time random processes (unless otherwise noted) on a filtered probability space $(\Omega, \mathcal F, (\mathcal F_t)_{t \ge 0}, \mathbb P)$ with the filtration satisfying the usual conditions. 

\subsection{Organization of the paper}  In Section \ref{sec-jump}, we study the model  in an interactive jump environment.  In Section \ref{sec-diffusive}, we study  the single-server queue with a reflected jump diffusion environment. In Section \ref{sec-rate}, we estimate the explicit rate of exponential convergence for the case of compact environment state space, for both models in  Sections  \ref{sec-jump} and \ref{sec-diffusive}. In Section \ref{sec-appendix} we state and prove some auxiliary lemmata. We make some concluding remarks in Section \ref{sec-conclusion}.

\section{$M/M/1$ queue in an interactive jump environment} \label{sec-jump}

Consider an $M/M/1$ queue with an infinite waiting space operating in an interactive jump environment described as follows. Let  $D$ be a finite or countable state space. For every $n \in \mathbb Z_+$, let $\mathbf{T}_n = (\tau_n(z, z'))_{z, z \in D}$ be the generator of an irreducible continuous-time Markov chain on $D$; this (finite or countable-sized) matrix is called {\it nominal jump intensity matrix} for the jump process $Z$ in the queueing state $n$. We define a two-component Markov process  $(N,Z)$ taking values in the countable state space $\ZZ_+ \times D$ with the following generator matrix $\mathbf{R}= \big( R[(n,z),(n',z')] \big)$:
\begin{align} 
\label{Rmatrix-MM1}
\begin{split}
R[(n,z),(n+1, z)] &=  \la(z),  \quad R[(n,z),(n-1, z)] =  \mu(z),   \\
R[(n,z),(n, z')] &=  \rho^{-n}(z) \tau_n(z,z'),  \quad R[(n,z),(n', z')] = 0,\quad n \ne n',\, z \ne z'.
\end{split}
\end{align}
where $\rho(z):= \la(z)/\mu(z)$ for each $z \in D$. 
Here $N=\{N(t): t\ge 0\}$ represents the number of jobs in the system (including those in queue and in service), taking values in $\ZZ_+$, and $Z= \{Z(t): t \ge 0\}$ represents a jump process taking values in  $D$. 
When the environment is in state $z$, the arrival and service rates for the queueing process are $ \la(z)$ and $ \mu(z)$, respectively, both depending on state $z$. 

When the queue size is in state $n$, the transition of the environment $Z$ from state $z$ to state $z'$ occurs at the rate $ \rho^{-n}(z) \tau_n(z,z')$. Note that the fourth equation in~\eqref{Rmatrix-MM1} does not allow simultaneous jumps for $N$ and $Z$. It is evident that the pair $(N,Z)$ is a well-defined Markov process in $\ZZ_+\times D$ with the generator $\mathbf{R}$.

\begin{rmk} 
We do not multiply this transition rate $\tau_n$ by a factor $\beta_n$: Dependence on $n$ is already enshrined in the rate $\tau_n$. 
We impose a condition~\eqref{as-MM1-T2} to guarantee the product form of the steady state. 
\end{rmk}

We first make the following assumption on the 
nominal jump intensity matrix $\mathbf{T}_n$. 

\begin{asmp}  \label{as-MM1-T}
For each $n\in \ZZ_+,\, z \in D$, and for some function $v: D \ra \RR_+$,
\beql{as-MM1-T2}
v(z)\sum_{z'\in D} \tau_n(z,z') = \sum_{z'\in D} v(z') \tau_n(z',z). 
\eeq
\end{asmp} 

For fixed $n\in \ZZ_+$, if we define a Markov process $\tilde{Z}_n :=\{\tilde{Z}_n(t): t\ge 0\}$ on $D$ with the nominal jump intensity matrix $\mathbf{T}_n$ as the generator, then~\eqref{as-MM1-T2} implies that $v(\cdot)$ defines an invariant measure for $\tilde{Z}_n$. If $\sum_{z\in D} v(z)< \infty$, then this measure can be normalized to a probability distribution. If 
\beql{}
\sum_{z'\in D} \tau_n(z,z') = \sum_{z'\in D} \tau_n(z',z),  \non
\eeq
then the counting measure is invariant for $\tilde{Z}_n$; if $D$ is a finite set, then it is normalized to a uniform distribution on $D$. It is important to note that the invariant measure $v(\cdot)$ does not depend on $n$, although the jump intensity matrix $\mathbf{T}_n$ depends on $n$.

\begin{rmk}
A simplest example is when $\tau_n(z,z')$ has a multiplicative form: 
$$
\tau_n(z,z') = \beta_n \tau(z,z')
$$
for some transition rate matrix $\tau(z,z')$ satisfying $v(z)\sum_{z'\in D} \tau(z,z') = \sum_{z'\in D} v(z') \tau(z',z)$. However, we provide  examples below in which $\tau_n(z,z')$ depends on $n$ in a nontrivial manner while the existence of $v$ independent of $n$ is guaranteed. See Examples  \ref{ex-discrete-1} and  \ref{ex-discrete-2}. 
 \end{rmk}

\begin{asmp}  \label{as-MM1}
The functions $ \rho, v$ satisfy
\beql{as-MM1-2}
\rho(z) <1  \qforq z \in D,
\eeq
\beql{as-MM1-3}
\Xi:= \sum_{z\in D}  \frac{v(z)}{1-\rho(z)} = \sum_{n=0}^{\infty} \sum_{z\in D} \rho^n(z) v(z) < \infty.
\eeq
\end{asmp} 
Note that the constant $\Xi$ is the normalization constant in the joint invariant measure $\pi$ in \eqref{MM1-pi}.

\begin{theorem} \label{thm-MM1}
Under Assumptions \ref{as-MM1-T} and \ref{as-MM1}, the Markov process $(N,Z)$ is irreducible, aperiodic, and positive recurrent. It  has
an invariant probability measure
\beql{MM1-pi}
\pi(n,z) := \eta(n,z)/\Xi,  \quad \forall (n,z) \in \ZZ_+\times D,
\eeq
where $\Xi$ is given in \eqref{as-MM1-3}, and 
\beql{MM1-eta}
\eta(n,z) :=  \rho^n(z) v(z), \quad \forall (n,z) \in \ZZ_+\times D,
\eeq
This process has transition kernel $P_t(x, \cdot) $ which converges to this invariant measure:
\begin{equation}
\label{eq:ergodic-discrete}
\norm{P_t(x, \cdot) - \pi(\cdot)}_{\TV} \to 0\ \mbox{as}\ t \to \infty,\ \mbox{for all}\ x \in \ZZ_+\times D.
\end{equation}
\end{theorem}

\medskip 
\noindent\emph{Proof.} 
We first show that the process $(N,Z)$ is irreducible and aperiodic. It follows from the observation that for every $t > 0$, $(n, z), (n', z') \in \mathbb Z_{+}\times D$, one can with positive probability get from $(n, z)$ to $(n', z')$ in time $t$. For the measure $\eta$ from~\eqref{MM1-eta} to be finite, we need
$$
\sum_{(n,z)} \eta(n,z) = \sum_{(n,z)} \rho^n(z) v(z) = \sum_z \frac{v(z)}{1-\rho(z)} < \infty, 
$$
which is implied by \eqref{as-MM1-2}--\eqref{as-MM1-3} in Assumption \ref{as-MM1}. If we  prove that $\eta$ from~\eqref{MM1-eta} is indeed an invariant measure, the positive recurrent  property follows from \cite[Theorem 3.5.3]{Norris}, \cite[Theorem 2.7.18]{SK}, and then the ergodicity, as in~\eqref{eq:ergodic-discrete}, follows from \cite{MT1993a}. To verify that $\eta(n,z)$ in \eqref{MM1-eta} is an invariant measure, we show that $\eta' \mathbf{R} =0$. Let us show that for all $n = 1, 2, \ldots$ and $z \in D$, 
\begin{align} 
\label{MM1-p}
\begin{split}
-  \eta(n,z) R[(n,z),(n,z)] & =  \eta(n-1,z) R[(n-1,z),(n,z)] + \eta(n+1,z) R[(n+1,z),(n,z)]\\  & \qquad \qquad  \qquad \qquad+ \sum\nolimits_{z'\neq z} \eta(n,z') R[(n,z'),(n,z)], \\ - \eta(0,z) R[(0,z),(0,z)] & =   \eta(1,z) R[(1,z),(0,z)] + \sum\nolimits_{z'\neq z} \eta(0,z') R[(0,z'),(0,z)].  
\end{split}
\end{align}
By \eqref{Rmatrix-MM1}, the left- and right-hand sides of the first equation in \eqref{MM1-p} are equal to, respectively,
\begin{align*}
\eta(n,z)& \sum_{(n',z')\neq (n,z)} R[(n',z),(n',z)]  \\
& = \rho^n(z) v(z) \bigg(R[(n,z),(n+1,z)] + R[(n,z),(n-1,z)] + \sum_{z'\neq z} R[(n,z),(n,z')]  \bigg)  \\
& = \rho^n(z) v(z) \bigg(  \la(z) +  \mu(z) + \sum_{z'\neq z} \rho^{-n}(z)\tau_n(z,z') \bigg) \\
& =\rho^n(z)v(z)(\la(z) + \mu(z)) + v(z) \sum_{z'\neq z} \tau_n(z,z');\\
\rho^{n-1}(z) &v(z) \la(z) + \rho^{n+1}(z) v(z)  \mu(z) + \sum_{z'\neq z} \rho^{n}(z') v(z') \rho^{-n}(z') \tau_n(z',z) \\
& = \la(z) v(z) \rho^n(z)(\la(z) + \mu(z)) + \sum_{z'\neq z} v(z')\tau_n(z',z). 
\end{align*}
From~\eqref{as-MM1-T2} in Assumption \ref{as-MM1-T}, the last terms in the right-hand side of these two last equations are equal. This proves the first equation in~\eqref{MM1-p}; the second one is similar. This completes the proof.   

\medskip

\begin{exm}  \label{ex-discrete-1}
\emph{($D$ as a union of  finite sets)} In Examples~\ref{ex-discrete-1} and~\ref{ex-discrete-2}, $\delta (i,j)$ stands for the
Kronecker delta.  
Given $n\in{\mathbb Z}_+$, let $D_n$ be a finite set in $(0,1)$ with cardinality $m_n$. 
For definiteness, assume that $1<m_n<M$ where $M\in{\mathbb Z}_+$ is a fixed value. 
Introduce an enumeration of points in each $D_n$: $D_n=\{z(1),\ldots ,z(m_n)\}$
(say, in a increasing order) and make a convention that $z(0)=z(m_n)$, $z(m_n+1)=z(1)$.
Sets $D_n$ can have common points for different $n$ or be pair-wise disjoint.  Set 
$D=\operatornamewithlimits{\cup}\limits_{n}D_n$ and $\upsilon (z)=1$ for $z\in D$. Set $D$ can be 
finite or countable.

Next, take a subset ${\mathbb L}\subseteq{\mathbb Z}_+$ (${\mathbb L}$ or 
${\mathbb Z}_+\setminus{\mathbb L}$ can be empty). For $n\in\mathbb L$, set
$$
\tau_n(z,z') = \frac{\beta_n}{m_n-1}, \quad \forall\;z,z'\in D_n\;\hbox{ with }\;z \neq z'.
$$
For $n\in{\mathbb Z}_+\setminus{\mathbb L}$, set
$$ \tau_n(z(i),z(j)) = \frac{1}{2}\delta (j,i\pm 1),\quad\forall\;i,j\in\{1,\ldots ,m_n\}.$$
Here $\beta_n\in (0,\infty)$ are scaling constants depending on $n$ (which is irrelevant for the invariant 
measure of the process $(N,Z)$). Pictorially, $\tau_n$ for $n\in\mathbb L$ describes uniform jumps
on $D_n$ while for $n\in{\mathbb Z}_+\setminus{\mathbb L}$, $\tau_n$ yields a `nearest-neighbor' 
walk with cyclic (periodic) boundary condition.
Either way, the counting measure $\upsilon$ is invariant; cf. Assumption~\ref{as-MM1-T}. Thus, \eqref{as-MM1-T2} holds true.

Then $\mathbf{T}_n= \big(\tau_n(z,z')  \big)$ generates a Markov chain $\tilde{Z}_n$ with an invariant 
probability measure ${\mathbf 1}_{D_n}(z)/m_n$,  $z \in D$.  The invariant measure 
$\eta$ is then given  in \eqref{MM1-eta} with $\eta (n,z)=z^n$. 
\end{exm}

\begin{exm}  \label{ex-discrete-2}
\emph{($D$ as a countable set, $\tau_n$ as a null-recurrent jump chain.)}
Assume that  $D \subset (0,1)$ is countable, and can be enumerated by $i=0,\pm 1,\pm 2$, 
so that $\rho_i:= \rho_{z_i}$ for $i\geq 0$ satisfies $\rho_0 < \rho_1 < \cdots < 1$ and 
$\lim_{i \rightarrow \infty} \rho_i =1$. (Enumeration with labels $i=-1,-2,\ldots$
does not matter.) Set $\ups (z)=1$ and 
$$
\tau_n(z_0,z_j) = \tau_n(z_i, z_{i+j})  = \beta_n\delta (j,n), \quad \forall i,j\in{\mathbb Z}.
$$
Here, as earlier, $\beta_n$ is a scaling constant depending on $n$ (again irrelevant for the invariant 
measure of the process $(N,Z)$). 
Then  $\mathbf{T}_n= \big(\tau_n(z,z')  \big)$ generates a null-recurrent Markov chain $\tilde{Z}_n$ with  
the invariant measure $\upsilon (z)= 1$, $z \in D$. 
Thus, the random traffic intensity $\rho_{\tilde{Z}_n}$, depending on both the state of the queue and the environment, will approach the critical value $1$ infinitely often. However, under the condition \eqref{as-MM1-3} the resulting Markov process $(N,Z)$ is positive recurrent, with an invariant measure $\eta(n,z) = \rho^n(z)$ for $(n,z) \in \ZZ_+\times D$.  
\end{exm}

\medskip

\section{$M/M/1$ queue in an interactive diffusive environment} 
\label{sec-diffusive}

\subsection{Reflected jump-diffusions} In this section we consider the queue with $\lambda$ and $\mu$ dependent on a diffusive environment process $Z(t)$. 
First, let us define the dynamics of this environment process as a reflected (jump) diffusion in a certain domain in $\mathbb R^d$.

It is instrumental to recapitulate some basic notion. 
A {\it domain} in $\RR^d$ is the closure of an open connected subset. A domain $D$ is called {\it smooth} if its boundary $\partial D$ is a $(d-1)$-dimensional $C^2$ manifold. Take $m$ smooth domains $D_1, \ldots, D_m$ in $\RR^d$. Assume  $D = \cap_{i=1}^m D_i$ has boundary $\partial D$ with $m$ {\it faces:} $F_i := \partial D\cap\partial D_i$ which are $(d-1)$-dimensional manifolds with an edge, and such that all $m$ domains are essential: Removal from the intersection of any domain will change the result. Then $D$ is called a {\it piecewise smooth domain} in $\RR^d$. Define by $\mathbf{n}_i(z)$ the inward unit normal vector to $\partial D_i$ at $z \in F_i$. {\it Inward} in this case is defined as pointing inside $D_i$, even if this is not inside $D$. An important example is a {\it convex polyhedron} with $D_i$ being half-spaces. Of particular interest is the positive orthant $D = \RR^d_+$. Of course, smooth domains also belong to this class of domains, with $m = 1$. 

Take continuous functions $g : D \to \mathbb R^d$ and $\Sigma : D \to \mathbb R^{d\times d}$ such that the matrix $\Sigma(z) = (a_{ij}(z))$ is symmetric and positive definite for all $z \in D$, and there exists a $\delta > 0$ such that $\Sigma(z)v\cdot v \ge \delta\norm{v}^2$ for all $v \in \mathbb R^d$ and $z \in D$. For every $z \in D$, define a finite measure $\varpi(z, \cdot)$ on $D$ such that $\varpi(z, \cdot) \Rightarrow \varpi(z^0, \cdot)$  as $z \to z^0$ in $D$. Recall that $\Rightarrow$ denotes weak convergence.  Take $r_i : F_i \to \mathbb R^d$: continuous functions, pointing inside $D$; that is, $r_i(z)\cdot\mathbf{n}_i(z) > 0$ for $i = 1, \ldots, m,\, z \in F_i$. Let us define a {\it reflected jump-diffusion}: a process $Z = \{Z(t):\, t \ge 0\}$ in $D$ with {\it drift vector field} $g$, {\it diffusion matrix field} $\Sigma$, {\it jump measures} $\varpi(z, \cdot)$ and {\it reflection vector fields} $r_1, \ldots, r_m$. 

This process will be adapted and right-continuous with left limits. Take a $d$-dimensional Brownian motion $B = \{B(t):\, t \ge 0\}$, adapted to the filtration. Take continuous nondecreasing processes $\ell_i = \{\ell_i(t):\, t \ge 0\}$ for $i = 1, \ldots, m$ such that $\ell_i$ can grow only when $Z(t) \in F_i$, another right-continuous process with left limits $Z = (Z(t),\, t \ge 0)$ with values in $D$,  and yet another process $\mathcal{N} = \{\mathcal{N}(t): t\ge 0\}$ which is right-continuous piecewise constant, with jump measure $\varpi(Z(t-), \cdot)$, and such that 
\begin{equation}\label{E-processZ}
\mathrm{d} Z(t) = g(Z(t))\,\mathrm{d} t + \Sigma^{1/2}(Z(t))\,\mathrm{d}B(t) + \mathcal N(t) + \sum\limits_{i=1}^mr_i(Z(t))\,\mathrm{d}\ell_i(t),\,\quad t \ge 0.
\end{equation}
We assume the equation~\eqref{E-processZ} has a well-defined unique weak solution, and forms a Feller continuous strong Markov semi-group, with generator 
\begin{equation} \label{calA-def}
\mathcal A f(z) = g(z)\cdot\nabla f(z) + \frac12\tr(\Sigma(z)\nabla^2f(z)) + \int_D(f(z') - f(z))\varpi(z, \mathrm{d} z'),
\end{equation}
which consists of a nondegenerate uniformly elliptic diffusion and a state-dependent finite jump measure. This existence and uniqueness were proved under Lpischitz conditions on vector field $g(\cdot)$ and the matrix $(a_{ij}(\cdot))$, as well as continuity of $r_i(\cdot)$ for each $i = 1, \ldots, m$, and some additional technical conditions. The case without jumps was proved in \cite{LS1984}; the general case follows from the standard construction by {\it piecing out}, \cite{Sawyer}. The reflection at the boundary translates into boundary conditions for~\eqref{calA-def}:
\begin{equation}
\label{eq:Neumann}
r_i(z)\cdot\nabla f(z) = 0,\,\quad z \in F_i,\, \quad i = 1, \ldots, m.
\end{equation}
The dynamics of this process can be described as follows:

\begin{itemize}
\item As long as it is strictly inside $D$, this process behaves as a jump-diffusion in $d$ dimensions with drift vector field $g$ diffusion matrix field $\Sigma$, and family $\varpi$ of jump measures. These jump measures are such that the process does not jump out of $D$. 

\item At a point $z \in F_i$, $i = 1, \ldots, m$, it is reflected back inside the domain $D$, according to the vector $r_i(z)$. 

\item If it hits the lower-dimensional edges: intersections of two or more faces $F_1, \ldots, F_m$, it is reflected back inside $D$ according to a positive linear combination of reflection vectors corresponding to these intersecting faces. 
\end{itemize}

Normal reflection corresponds to the case when $r_i(z) = \mathbf{n}_i(z)$, where $z \in F_i$ and $i = 1, \ldots, m$. 

\begin{rmk} In the case of a diffusion without reflection, the state space may be $\mathbb{R}^d$, or still some subset $D$. The latter happens if the drift coefficient is sufficiently large to compel the process to stay in a certain domain. An example of this is the drift for a Bessel process on the half-line, see \cite[Chapter 3, Problem 3.23]{KSBook}.
\end{rmk}

In the case $d = 1$, for a reflection on $[a, b]$, we have a normal reflection, and the boundaries consisting of two pieces $\{a\}$ and $\{b\}$. For a reflection on $[a, \infty)$, we have a normal reflection again, with the boundary $\{a\}$.

\subsection{Construction of the joint Markov process} \label{sec-diffusive-construction}
Let us now use symbol $z$ for a point in $D$ (instead of $x$). 
Take continuous functions $\lambda, \mu : D \to (0, \infty)$ with $\lambda(z) \le \mu(z)$ for $z \in D$. Define the {\it traffic intensity:}
\begin{equation}
\label{eq:intensity}
\rho(z) := \frac{\lambda(z)}{\mu(z)} \le 1.
\end{equation}
For every $z \in D$, consider an $M/M/1$ queue with arrival intensity $\lambda(z)$ and service intensity $\mu(z)$, where $n$ is the state of this queue. The process $\tilde{N}$ counting the number of jobs in the system, called the \emph{queueing process} in the sequel, is a continuous-time Markov process on $\mathbb Z_+$ with generator
\begin{equation} \label{eq:calMz}
\mathcal M_zf(n) = \lambda(z)(f(n+1) - f(n))  + 1_{\{n \ne 0\}}\mu(z)(f(n-1) - f(n)).
\end{equation}

\smallskip

We now consider a $(1+d)$-dimensional Markov process $(N, Z) = \{(N(t), Z(t)): t \ge 0\}$ with values in $\mathbb Z_+ \times D$ which evolves as follows:

\smallskip

(a) If $N(t) = n \in \mathbb Z_+$, then $Z$ behaves as a reflected jump-diffusion in $D$ with generator  $\rho^{-n}(z)\beta_n\mathcal A$ and reflection fields $r_1, \ldots, r_m$. 

\smallskip

(b) if $Z(t) = z$, then $N(t)$ jumps from $n$ to $n+1$ with intensity $\lambda(z)$, and (if $n \ne 0$) to $n-1$ with intensity $\mu(z)$.

\smallskip

 Here $\beta_n$ is the \emph{variability coefficient} for the diffusive environment, depending on the queueing state $n$, and  $\rho^{-n}(z)$ is the \emph{queueing impact} factor, capturing the impact from the traffic intensity (congestion) from the queueing process.  
The component $N$ can be informally described as the queueing process of an $M/M/1$ queue with  arrival and service rates, $\lambda(z)$ and $\mu(z)$, respectively. These rates depend on an auxiliary process $Z$. The dynamics of $Z$, however, depends on the current position of this queueing process. Therefore, we call such a system as \emph{an $M/M/1$ queue in an interactive diffusive environment.} 

\smallskip

The joint dynamics  is described via a combined Markov process $(N, Z)$ with the following generator: 
\begin{align}
\label{eq:generator}
\mathcal Lf(n, z)  = \mathcal M_zf(n, z) + \beta_n\rho^{-n}(z)\mathcal A f(n, z), \quad f \in \mathcal D.
\end{align}
Here, $\mathcal D$ stands for the following subspace of the domain of $\mathcal L$:
\begin{align} \label{calD-def}
\mathcal D & := \{f : \mathbb Z_+\times D \to \mathbb R\mid \forall n \in \mathbb Z_+\,, \, f(n, \cdot) \in \mathcal D_D\},\\
\mathcal D_D & := \{f \in C^2_b(D)\mid r_i(z)\cdot \nabla f(z) = 0,\,\, z \in F_i,\, i = 1, \ldots, m\}. \nonumber 
\end{align}
(Note that we were intentionally loose on the domains of $f$ in \eqref{calA-def} and \eqref{eq:calMz}, but they are clear from this definition.) From the general theory of {\it piecing out} it follows that this is a Feller process, see \cite{Sawyer}. We denote by $C^2_b(D)$  the set of twice continuously differentiable functions $D \to \mathbb R$ which are bounded with their first and second derivatives (the last condition, automatically fulfilled for bounded $D$). This is a separable Banach  space with the norm 
$$
\norm{f}_{D, 2} := \sup\limits_{z\in D}\left(|f(z)| + \norm{\nabla f(z)} + \norm{\nabla^2f(z)}\right). 
$$
Denote by $P^t(y, \cdot)$ the transition kernel of $(N, Z)$ where $y = (n, z) \in \mathbb Z_+\times D$. 

\smallskip

We give three special cases to illustrate the construction above. 

\smallskip
\begin{itemize}
\item[(a)]
 \emph{$M/M/1$ queue with an interactive diffusive arrival rate.}

\smallskip

\noindent Assume that $\lambda(z)=z$ and $\mu\equiv1$. 
Let $D=[0,1]$, and the generator $\mathcal{A}$ in \eqref{calA-def} be that of a  reflected diffusion in $(0,1)$ without jumps. 
The reflections at $0$ and $1$ correspond to the Neumann boundary conditions:
$$
 \frac{\partial}{\partial z} f(n, 0+) =0, \quad  \frac{\partial}{\partial z} f(n, 1-) = 0, \quad \forall n \in \ZZ_+. 
$$

\smallskip

\item[(b)] \emph{$M/M/1$ queue with an interactive diffusive service rate.}

\smallskip
\noindent Assume that $\lambda\equiv 1$ and $\mu(z) = z$. 
Let $D=[\mu_0,\infty)$ for some $\mu_0 \ge 1$, and the generator $\mathcal{A}$ in \eqref{calA-def} be that of a simple RBM on $[\mu_0, \infty)$ without jumps. The reflection at $\mu_0$ satisfies the Neumann boundary condition.

\smallskip

\item[(c)] \emph{$M/M/1$ queue with both diffusive arrival and service rates.}
\smallskip

\noindent Take $D = \{(z_1,z_2)\in \RR^2_+: z_2\le z_1 \}$ be a cone in the positive orthant, and the generator  $\mathcal{A}$ in \eqref{calA-def}  be that of a two-dimensional  Brownian motion in $D$ with normal reflections at the boundary.  
Let $(\lambda(z), \mu(z)) = z$. Then the arrival and service rates of the $M/M/1$ queue follow the dynamics of a reflected RBM in $D$ in the interactive manner described above. 
\end{itemize}

\subsection{Invariant measures} We need the following assumptions on some properties of the reflected jump-diffusion process. The first assumption states that for each level, there exists a steady-state distribution. The second assumption ensures that the whole process has a steady-state distribution.

\begin{asmp} 
\label{asmp:diff}
Assume the (reflected) jump-diffusion with generator $\mathcal A$ is positive recurrent, and has a unique stationary/invariant measure $\nu_D$, together with {\it boundary measures} $\nu_{F_i},\, i = 1, \ldots, m$. This means that the stationary copy of this process $\tilde{Z}^* = \{\tilde{Z}^*(t):\, t \ge 0$\} with $\tilde{Z}^*(t) \sim \nu_D$ for $t \ge 0$, satisfies the following condition: For every $t \ge 0$, each $i = 1, \ldots, m$, and every bounded function $f : F_i \to \mathbb R$, 
\begin{equation}
\label{eq:def-b-mes}
\mathbb E\int_0^tf(\tilde{Z}^*(s))\,\mathrm{d}\ell_i(s) = t\,\int_{F_i}f(z)\,\nu_{F_i}(\mathrm{d}z),
\end{equation}
where $\ell_i(s)$ is the nondecreasing process in \eqref{E-processZ}. 
\end{asmp}

\begin{asmp}
\label{asmp:finite}
The measure $\nu_D(\cdot)$ satisfies 
\begin{equation}
\label{eq:total}
\Xi := \int_D\frac{\nu_D(\mathrm{d}z)}{1 - \rho(z)} = \sum_{n=0}^{\infty} \int_D \rho^n(z) \nu_D(\mathrm{d}z) \,< \infty.
\end{equation}
\end{asmp}

Note that $\Xi$ is the normalization constant in the joint invariant measure of $(N,Z)$ in \eqref{eq:invariant-measure}. This invariant measure on each boundary $F_i$ has value zero.

\begin{rmk}
Similarly to~\eqref{eq:def-b-mes}, we can define the concept of {\it boundary measures} for the joint process $(N, Z)$. First, construct the boundary process $\ell_i = (\ell_i(t),\, t \ge 0)$ for the component $Z$ and face $F_i$ of the boundary $\partial D$. Assume $0 = \rho_0 < \rho_1 < \ldots$ are jump times for $N$. Then $Z(\rho_k+t)$  for $t \in [0, \rho_{k+1} - \rho_k]$ behaves as a reflected jump-diffusion on $D$ with generator $\rho^{-n_k}(z)\beta_{n_k}\mathcal A$ and reflection fields $r_1, \ldots, r_m$, with $N(t) = n_k$ for $t \in [\rho_k, \rho_{k+1})$. Thus there exist a continuous nondecreasing process $\ell_i^{(k)}(t),\, t \in [0, \rho_{k+1} - \rho_k]$ such that~\eqref{E-processZ} holds with adjusted drift vector field, diffusion matrix field, and jump measures family. Define
$$
\ell_i(t) = \ell_i(\rho_k) + \ell_i^{(k)}(t-\rho_k),\quad t \in [\rho_k, \rho_{k+1}],
$$
using induction over $k$. This defines $\ell_i = (\ell_i(t),\, t \ge 0)$ for $i = 1, \ldots, m$. Next, define a {\it boundary measure} $\nu_{F_i}$ on the face $F_i$ corresponding to a stationary distribution $\pi$ for this joint process $(N, Z)$: Take the corresponding stationary copy $(N^*, Z^*)$ with $(N^*(t), Z^*(t)) \sim \pi$ for $t \ge 0$. For a bounded function $f : \mathbb Z_+\times F_i \to \mathbb R$ and a $t \ge 0$, 
\begin{equation*}
\mathbb E\int_0^tf(\tilde{N}^*(s), \tilde{Z}^*(s))\,\mathrm{d}\ell_i(s) = t\,\sum\limits_{n=0}^{\infty}\int_{F_i}f(n, z)\,\nu_{F_i}(\{n\}\times\mathrm{d}z).
\end{equation*}
\label{rmk:boundary-measures-joint}
\end{rmk}

Now we are ready to state and prove the main result of this section. 

\begin{thm} 
Under Assumptions~\ref{asmp:diff} and~\ref{asmp:finite}, there is a unique invariant  measure for $(N, Z)$:
\begin{equation}
\label{eq:invariant-measure}
\pi(\{n\}, \mathrm{d}z) = \Xi^{-1}\rho^n(z)\nu_D(\mathrm{d}z).
\end{equation}
The corresponding boundary measures $\pi_i$ for $F_i$ (if there is reflection) are given by
\begin{equation}
\label{eq:boundary-measure}
\pi_i(\{n\}, \mathrm{d}z) = \Xi^{-1}\rho^n(z)\nu_{F_i}(\mathrm{d}z),\quad i = 1, \ldots, m.
\end{equation}
Finally, this Markov process is ergodic: for every $y \in \mathbb Z_+\times D$, 
\begin{equation}\label{eq:long-term-conv}
\norm{P^t(y, \cdot) - \pi(\cdot)}_{\TV} \to 0\quad \mbox{as}\quad t \to \infty.
\end{equation}
\label{thm:inv-measure}
\end{thm}

\begin{proof} From stationarity we immediately get: for all $f \in C^2_b(D)$, 
\begin{equation}
\label{eq:BAR}
\int_D\mathcal A f(z)\,\nu_D(\mathrm{d} z) + \sum\limits_{i=1}^m\int_{F_i}r_i\cdot \nabla f(z)\,\nu_{F_i}(\mathrm{d} z) = 0.
\end{equation}
This is called the {\it basic adjoint relationship} in the literature. We refer to \cite{WilliamsSurvey} for its deduction in the case of a convex polyhedron; the same is true for a general piecewise smooth domain $D$, as in our case. Apply \cite[Theorem 1.7, Theorem 2.2, Lemma 2.4, Remark 2.5]{KurtzStockbridge2001} using their notation, with the state space $E = \mathbb Z_+\times D$; $U = \{0, 1, \ldots, m\}$, where the point $0$ corresponds to the domain $D$ itself, and $i = 1, \ldots, m$, correspond to faces $F_1, \ldots, F_m$ of the boundary; for all $z \in D$, $n \in \mathbb Z_+$, and $u \in U$, 
\begin{align}
\label{eq:boundary-techniques}
\begin{split}
\mu_0(\{u\}\times\{n\}\times \mathrm{d}z) & = 1(u = 0)\,\rho^n(z)\nu_D(\mathrm{d}z),\\   \mu_1(\{u\}\times \{n\}\times \mathrm{d}z) & = 1(u \ne 0)\,\rho^n(z)\nu_{F_u}(\mathrm{d}z);\\ \mu^E_0(\{n\}\times\mathrm{d}z) &  = \rho^n(z)\nu_D(\mathrm{d}z), \\  \mu^E_1(\{n\}\times\mathrm{d}z) & = \rho^n(z)\left[\nu_{F_1}(\mathrm{d}z) + \ldots + \nu_{F_m}(\mathrm{d}z)\right];\\ \eta_0((n, z), \{u\}) & = 1(u = 0),  \\   \eta_1((n, z), \{u\}) & = 1(u \ne 0);\\ Af((n, z), u) &  := \mathcal Lf(n, z), \quad \text{cf.} \, \eqref{eq:generator}, \\   Bf((n, z), u) & := 1(u \ne 0,\, z \in \partial D)\, \beta_n\rho^{-n}(z)r_u(z)\cdot\nabla f(z).
\end{split}
\end{align}
We need to check \cite[Condition 1.2]{KurtzStockbridge2001} on the absolutely continuous generator $A$ and the singular generator $B$. Let
$$
\mathcal D := \{f : \mathbb Z_+ \times D \to \RR \mid \forall n \in \mathbb Z_+,\, f(n, \cdot) \in C^2_b(D)\}.
$$
Part (i) requires that $A,B: \mathcal{D} \subset C_b(E) \to C(E\times U)$, and the unity function $\mathbf{1}(n,z) = 1$ for $(n,z) \in E$ satisfies $\mathbf{1}\in \mathcal{D}$, $A\mathbf{1}=0$, and $B\mathbf{1}=0$. This is trivially satisfied. 

Part (ii) requires that there exist $\psi_A(n,z)$ and $\psi_B(n,z)$ in $C(E\times U)$, $\psi_A,\psi_B \ge 1$ and constants $a_f, b_f$, $f\in \mathcal{D}$ such that 
$$
|Af(x,u)| \le a_f \psi_A(x,u),\quad |Bf(x,u)| \le b_f \psi_B(x,u), \quad \forall (x,u) \in \mathcal{U}
$$
where $\mathcal{U}$ is any closed set of $E\times U$. We can take $a_f = b_f := \norm{f}_{D, 2}$, and 
$$
\psi_A = \psi_B = \norm{A(z)} + \sum_{i=1}^m\norm{r_i(z)} + \rho^n(z).
$$
Part (iii) requires the following: Defining $(A_0,B_0) = \{(f, \psi_A^{-1} A f, \psi_B^{-1} Bf): f \in \mathcal{D}\}$, $(A_0,B_0)$ is separable in the sense that there exists a countable collection $\{g_k\}\subset \mathcal{D}$ such that $(A_0,B_0)$ is contained in the bounded, pointwise closure of the linear span of 
 $\{(g_k, A_0g_k,B_0g_k) = (g_k,\psi_A^{-1}Ag_k,\psi_B^{-1} g_k)\}$.  
This is proved by taking a dense countable subset $\Upsilon$ of $C^2_b(D)$ in the norm $\norm{\cdot}_2$, and then taking a countable subset 
$$
\bigcup_{n=0}^{\infty}\mathcal [\Upsilon]^n \subseteq \mathcal D \simeq \left[C^2_b(D) \right]^{\mathbb Z_+}.
$$
This subset is dense in the sense of pointwise convergence.

Part (iv) requires that for each $u \in U$, the operators $A_u$ and $B_u$ defined by $A_uf(x)=Af(x,u)$ and $B_uf(x) = Bf(x,u)$ are pre-generators. This follows from \cite[Remark 1.1]{KurtzStockbridge2001}, because all these operators satisfy the positive maximum principle.  

Part (v) requires that $\mathcal{D}$ is closed under multiplication and separates points.  This follows directly from the definition.

Finally, we need to prove the main condition as in \cite[Theorem 1.7, (1.17)]{KurtzStockbridge2001}:
\begin{equation}
\label{eq:main-BAR}
\int_{E\times U}Af(x, u)\,\mu_0(\mathrm{d}x\times\mathrm{d}u) + \int_{E\times U}Bf(x, u)\,\mu_1(\mathrm{d}x\times\mathrm{d}u) = 0.
\end{equation}
From~\eqref{eq:boundary-techniques} and~\eqref{eq:generator}, canceling $\beta_n$ and $\rho^n(z)$ when appropriate, we rewrite the left-hand side of~\eqref{eq:main-BAR} as follows:
\begin{align}
\label{eq:final-BAR}
\begin{split}
& \sum\limits_{n=0}^{\infty}\beta_n \left[ \int_D\mathcal A f(n, z)\,\nu_D(\mathrm{d} z) + \sum\limits_{i=1}^m\int_{F_i}r_i(z)\cdot\nabla f(n, z)\,\nu_{F_i}(\mathrm{d}z) \right] \\  & \qquad +   \int_D\sum\limits_{n=0}^{\infty}\rho^n(z)\mathcal M_zf(n, \cdot)\,\nu_D(\mathrm{d} z).
\end{split}
\end{align}
The first line in~\eqref{eq:final-BAR} is equal to zero; this follows from~\eqref{eq:BAR}. Let us show that the second line in~\eqref{eq:final-BAR} is equal to zero, too. For every $z \in D$, $\mathcal M_z$ is the generator of the $M/M/1$ queue with arrival and service rates $\lambda(z)$ and $\mu(z)$. This queue has geometric stationary distribution $(1 - \rho(z))\rho^n(z),\, n \in \mathbb Z_+$. Thus 
\begin{equation}
\label{eq:Q}
\sum\limits_{n=0}^{\infty}\rho^n(z)\mathcal M_zf(n, \cdot) = 0,\quad z \in D.
\end{equation}
Integrating~\eqref{eq:Q} with respect to $\mu_D(\mathrm{d}z)$, we get: The second line in~\eqref{eq:final-BAR} is equal to zero. We interchanged integration and series, which we can do by uniform boundedness of $f$ combined with Assumption~\ref{asmp:finite}. This completes the proof of~\eqref{eq:main-BAR}, and with it \cite[(1.17)]{KurtzStockbridge2001}. Next, $K_1 := \partial D$ is the closed support for $\mu_1^E$. By \cite[Remark 2.5]{KurtzStockbridge2001}, the results of \cite[Lemma 2.4]{KurtzStockbridge2001} hold, and we can apply \cite[Theorem 2.2 (f)]{KurtzStockbridge2001}, and obtain the stationary copy of our process $(N, Z)$. 

\smallskip

We have written the proof for reflected diffusions. For non-reflected ones, it is simpler: we can simply verify~\eqref{eq:BAR}, which in our case then becomes
\begin{equation}
\label{eq:BAR-non-reflected}
\int_D\mathcal A f(z)\,\nu_D(\mathrm{d} z) = 0,\quad f \in C^2(D). 
\end{equation}
This is done similarly to the computation above, but without all boundary terms. The lack of reflection obviates the need to apply results cited above from \cite{KurtzStockbridge2001}. 

\smallskip

Finally, ergodicity follows from \cite[Theorem 6.1]{MT1993a} in the following way (for terminology, we refer the reader to this cited article \cite{MT1993a}).  Our process is positive Harris recurrent, since the invariant measure is finite. Meanwhile, every skeleton chain is irreducible, because of the following {\it irreducibility property}. Define a Lebesgue measure on $\mathbb Z_+\times D$ as a sum of Lebesgue measures on each layer of this set. 

\begin{lemma} For every $n \in \mathbb Z_+,\, z \in D$, and a subset $G \subseteq \mathbb Z_+\times D$ of positive Lebesgue measure, 
\begin{equation}
\label{eq:pvity}
P^t((n, z), G) > 0.
\end{equation}
\label{lemma:pve}
\end{lemma}

\begin{proof} Without loss of generality, assume $G = \{m\}\times E$ for a subset $E \subseteq D$ of positive Lebesgue measure, and $m \ge n$. We prove the statement~\eqref{eq:pvity} by induction over $m$. 

\smallskip

{\it Induction Base:} $m = n$. Consider the probability 
$$
P^t(y, G) = \mathbb P_{(n, z)}(N(t) = n,\, Z(t) \in E)
$$
that, starting from $y = (n, z)$, the joint process $(N, Z)$ at time $t$ will be in $\{n\}\times E$. This probability is bounded from below by
\begin{equation}
\label{eq:P-Q}
P^t(y, G) \ge Q^t_n(z, E) := \mathbb P_{(n, z)}\bigl(Z(t) \in E,\, N(s) = n,\, \forall\, s \in [0, t]\bigr).
\end{equation}
This probability $Q^t_n(z, E)$, in turn, is estimated from below by (with $z_* > 0$ fixed later):
\begin{align}
\label{eq:Q-tilde}
\begin{split}
Q^t_n(z, E) & \ge \tilde{Q}^t_n(z, z_*, E) := \mathbb P_{(n, z)}\bigl(Z(t) \in E,\,  \norm{Z(t)} \le z_*;\, N(s) = n,\, \forall\, s \in [0, t]\bigr)
\\ & \ge \exp\bigl(-t\max\limits_{\norm{z} \le z_*}(\lambda(z) + \mu(z))\bigr)\cdot q_*.
\end{split}
\end{align}
Here, $q_*$ is the probability that, starting from $Z_n(0) = z$, the reflected jump-diffusion $Z_n$ in $D$ with generator $\rho^{-n}(z)\mathcal L$ and reflection vector fields $r_1, \ldots, r_m$ ends at $Z_n(t) \in E$ and $\norm{Z_n(s)} \le z_*$ for $s \in [0, t]$. It follows from known properties of reflected jump-diffusions with nonsingular covariance matrix $\Sigma(\cdot)$ that $q_* > 0$ for large enough $z_* > 0$. This, together with~\eqref{eq:P-Q} and~\eqref{eq:Q-tilde}, proves that 
\begin{equation}
\label{eq:base}
P^t(y, G) \ge Q^t_n(z, E) \ge \tilde{Q}^t_n(z, z_*, E) > 0.
\end{equation}
Thus we have proved the statement~\eqref{eq:pvity} for $m = n$. 

\smallskip

{\it Induction Step:} First, consider the case $m = n+1$. This probability $P^t(y, G)$ is estimated from below by the probability that for some time $\tau \in [0, t]$, the process $N$ will stay at level $n$, then  jump at time $\tau$ at level $n+1$ and  stay there until time $t$, and $Z(t) \in E$. If $\hat{\mu}$ is the distribution of $\tau$  (which is a positive measure on $[0, t]$), then 
\begin{equation}
\label{eq:pve}
P^t(y, G) \ge \int_0^t\int_DQ^s_n(y, \mathrm{d}w)\,Q^{t-s}_{n+1}(w, E)\,\hat{\mu}(\mathrm{d}s).
\end{equation}
It suffices to show that the double integral in the right-hand side of~\eqref{eq:pve} is positive. Indeed, from~\eqref{eq:base} we get: $Q^s_n(y, E') > 0$ for $E' \subseteq D$ of positive Lebesgue measure, and $\quad Q^{t-s}_{n+1}(w, E) > 0$. In addition, $\hat{\mu}$ is a positive measure on $[0, t]$. Use twice the observation that the integral of a positive function over a positive measure is positive, and complete the proof that the right-hand side (and therefore the left-hand side) in~\eqref{eq:pve} is positive. 

\smallskip

Assuming we proved~\eqref{eq:pvity} for $m = n+k$, $k \ge 0$, let us prove this for $m = n + k + 1$:
\begin{equation}
\label{eq:pve-k}
P^t(y, G) \ge \int_DP^{t/2}(y, (n+k, \mathrm{d}w))\,P^{t/2}((n+k, w), \{n+k+1\}\times E) > 0.
\end{equation}
This follows from the same logic: The function $P^{t/2}((n+k, w), \{n+k+1\}\times E)$ is positive by the previous part of the induction step, applied to $n+k$ instead of $n$, and to $n+k+1$ instead of $m = n+1$. The measure $P^{t/2}(y, (n+k, \mathrm{d}w))$  is positive by the induction hypothesis. 
This completes the proof of this lemma.
\end{proof}

Using Lemma~\ref{lemma:pve}, we have shown ergodicity as in~\eqref{eq:long-term-conv}. Earlier, we have proved~\eqref{eq:invariant-measure} and~\eqref{eq:boundary-measure}. Thus we have completed the proof of Theorem~\ref{thm:inv-measure}. 
\end{proof}

\def\rd{\rm d} \def\diy{\displaystyle} \def\cA{\mathcal A}

{\bf Remark 2.3.} The crucial property is that for each $n\in \ZZ_+,\, z \in D,\,t>0, V,V'\subseteq D$, 
\beql{as-MM1-D2}\begin{array}{l}\diy
\int_{D\times D}{\mathbf 1}(z\in V,z'\in V'){\tt p}^t(z,z') \nu_D(\rd z)\nu_D(\rd z')\\
\qquad\qquad = \int_{D\times D} {\mathbf 1}(z\in V,z'\in V') {\tt p}^t(z',z)\nu_D(\rd z)\nu_D(\rd z'). 
\end{array}\eeq
In fact, further generalizations depend on whether an analog of this equality can be established.
Here ${\tt p}^t$ stands for the transition density for the diffusion with generator $\cA$ in \eqref{calA-def}.

\medskip

\subsection{A more general setup}

We  offer a similar result under a more general feedback scheme. For $n = 1, 2, \ldots$, fix a piecewsie smooth domain $D_n \subseteq D$ with $m_n$ faces of the boundary $\partial D_n$:
\begin{equation}
\label{eq:faces-n}
F^{(n)}_1, \ldots, F^{(n)}_{m_n},
\end{equation}
and corresponding reflection vector fields 
\begin{equation}
\label{eq:refl-n}
r^{(n)}_i : F^{(n)}_{i} \to \mathbb R^d.
\end{equation}
For each $n \in \mathbb Z_+$, this domain $D_n$, its boundary $\partial D_n$ with faces~\eqref{eq:faces-n}, and reflection vector fields~\eqref{eq:refl-n} satisfy the same assumptions enunciated at the very beginning of Section~\ref{sec-diffusive}, as the original domain $D$ and reflection vector fields $r_1, \ldots, r_m$. In addition, we impose the following assumptions on domains $D$ and $D_n$. 

\begin{asmp}
\label{asmp:domains}
For all $n \in \mathbb Z_+$, $D_n\cap D_{n+1}$ contains an open subset of $D$; and $D = \cup_{n \in \mathbb Z_+}D_n$. 
\end{asmp}

For every level $N(t) = n$ of the queue-size component, the environment variable $z\in D$ is kept fixed when $z\in D\setminus D_n$ and follows a reflected jump-diffusion process in $D_n$ as in~\eqref{E-processZ} where parameters vary with $n$. In other words, the process ${\widetilde Z}_n$ lives in $D$ but its mechanism depends on $n$. The generator ${\mathcal A}_n$ of ${\widetilde Z}_n$ has the form
\begin{equation} \label{calAn-def}
{\mathcal A}_n f(z) = g(z)\cdot\nabla f(z) + \frac12\sum\limits_{i=1}^d\sum\limits_{j=1}^da_{ij}(z)\frac{\partial^2f(z)}{\partial z_i\partial z_j} + \int_{D_n}(f(z') - f(z))\varpi(z, \mathrm{d} z'),
\end{equation}
for $z\in D_n$, and $\mathcal A_nf(z) = 0$ for other $z$. 
The generator $\mathcal L$ of the joint process, instead of~\eqref{eq:generator}, has the following form: 
\begin{align}
\label{eq:generator-new-version}
\mathcal Lf(n, z)  = \mathcal M_zf(n, z) + \beta_n\rho^{-n}(z)\mathcal A_nf(n, z), \quad f \in \widetilde{\mathcal D}.
\end{align}
Here $\widetilde{\mathcal D}$ is the following domain, defined similarly to \eqref{calD-def}:
\begin{align} 
\label{calD-def-new-version}
&\;\widetilde{\mathcal D}\;\;\;:= \;\{f : \mathbb Z_+\times D \to \mathbb R\mid \forall n \in \mathbb Z_+\,, \, f(n, \cdot) 
\in \mathcal D^{(n)}\},\\
&\mathcal D^{(n)} := \{f : D\to{\mathbb R}\mid f\in C_b({\overline D}),\, \left.f\right|_{D_n}\in C^2_b(D_n)
\cap C^1_b({\overline D}_n),\\
&\qquad\qquad \; r^{(n)}_i(z)\cdot \nabla f(z) = 0,\,\, 
z \in F^{(n)}_i,\, i = 1, \ldots, m_n\}. \nonumber \end{align}

Note that in this setup, the dependence of the generator $\mathcal{A}_n$ on the queueing state $n$ is only through the domain $D_n$ while the drift vector field $g(\cdot)$, covariance matrix field $\Sigma(\cdot)$, and jump measure family $\varpi(\cdot, \cdot)$ are all independent of $n$; see also  Examples \ref{ex:2.1} and \ref{ex:2.2}.

\smallskip

Let us impose assumptions on $\mathcal A_n$, similar to Assumptions~\ref{asmp:diff} and~\ref{asmp:finite}. 

\begin{asmp} 
\label{asmp:diff-new-version}
For every $n \in \mathbb Z_+$, the above (reflected) jump-diffusion in $D_n$ has a unique invariant distribution $\nu^{(n)}_{D_n}$, with corresponding boundary measures $\nu^{(n)}_{F^{(n)}_i},\, i = 1, \ldots, m_n$. 
\end{asmp}

\begin{asmp}
\label{asmp:sameness}
There exists a finite measure $\upsilon$ on $D$ whose restriction $\nu^{(n)}_{D_n}$ on $D_n$ is a stationary measure for $\widetilde{Z}_n$, for every $n$. 
\end{asmp}

This independence of the invariant measure $\upsilon$ of $n$  is similar to Assumption \ref{as-MM1-T} in Section 3.

\begin{asmp}
\label{asmp:finite-new-version}
We have:
\begin{equation}\label{eq:total-new-version}
\Xi := \sum\limits_{n=0}^{\infty}\int_{D_n}\rho^n(z)\upsilon(\mathrm{d}z)\,< \infty.
\end{equation}
\end{asmp}

Under these assumptions, we obtain the following theorem, analogous to  Theorem~\ref{thm:inv-measure}. 

\begin{thm}\label{thm:2.2}
Under Assumptions~\ref{asmp:domains}--\ref{asmp:finite-new-version}, the combined proces $(N,Z)$ with the generator $\mathcal L$ from \eqref{calAn-def} has a unique invariant probability distribution $\pi$ given by 
\begin{equation}
\label{eq:invariant-measure-new-version}
\pi(\{n\}, \mathrm{d}z) = \Xi^{-1}\rho^n(z)\upsilon(\mathrm{d}z).
\end{equation}
The corresponding boundary measures $\nu_{F_i}$ for $F_i$ (if there is reflection) are given by
\begin{equation}
\label{eq:boundary-measure-new-version}
\nu_{F_i}(\{n\}, \mathrm{d}z) = \Xi^{-1}\rho^n(z)\nu^{(n)}_{F_i}(\mathrm{d}z),\quad i = 1, \ldots, m_n.
\end{equation}
Finally, this process is ergodic in the sense of~\eqref{eq:long-term-conv}.  
\end{thm}

\begin{proof} 
For the proof of the stationary measure, we proceed very similarly to the proof of Theorem \ref{thm:inv-measure},  except that we change~\eqref{eq:boundary-techniques} 
\begin{align}
\label{eq:boundary-techniques-new-version}
\begin{split}
\mu_0(\{u\}\times\{n\}\times \mathrm{d}z) & = 1(u = 0)\,\rho^n(z)\,\upsilon(\mathrm{d}z),\\   \mu_1(\{u\}\times \{n\}\times \mathrm{d}z) & = 1(u \ne 0)\,\rho^n(z)\,\nu^{(n)}_{F_u}(\mathrm{d}z);\\ \mu^E_0(\{n\}\times\mathrm{d}z) &  = \rho^n(z)\,\upsilon(\mathrm{d}z), \\  
\mu^E_1(\{n\}\times\mathrm{d}z) & = \rho^n(z)\nu^{(n)}_{F_i}(\mathrm{d}z), \quad z\in F_i, \, i=1,\dots,m;
\\ \eta_0((n, z), \{u\}) & = 1(u = 0),  \\   \eta_1((n, z), \{u\}) & = 1(u \ne 0);\\ Af((n, z), u) &  := \mathcal Lf(n, z), \quad \text{cf.} \, \eqref{eq:generator-new-version}, \\   Bf((n, z), u) & := 1(u \ne 0,\, z \in \partial D)\,\rho^{-n}(z)\,r_u(z)\cdot\nabla f(z).
\end{split}
\end{align}
To prove ergodicity as in~\eqref{eq:long-term-conv}, similarly to Theorem~\ref{thm:inv-measure}, we show an analogue of Lemma~\ref{lemma:pve}:

\begin{lemma} For all $n, m \in \mathbb Z_+$, $z \in D$, and a subset $G \subseteq \mathbb Z_+\times D$ of positive Lebesgue measure:
\begin{equation}
\label{eq:pvity-n}
P^t((n, z), G) > 0.
\end{equation}
\label{lemma:pve-adjusted}
\end{lemma}

\begin{proof} Similarly to Lemma~\ref{lemma:pve-adjusted}, without loss of generality, assume $G = \{m\}\times E$ for a subset $E \subseteq D$ of positive Lebesgue measure, and $m \ge n$. 

\smallskip

{\it Case (a).} $z \in D_n,\, E \subseteq D_m$. We prove this statement similarly to Lemma~\ref{lemma:pve}, using induction over $m$. Induction base ($m = n$): can be shown as in~\eqref{eq:P-Q}. Induction step: for $m = n + 1$ we prove this as in~\eqref{eq:pve} (using the same notation), but we integrate over $D_n\cap D_{n+1}$ instead of $D$:
$$
P^t((n, z), \{m\}\times E) \ge \int_0^t\int_{D_n\cap D_{n+1}}Q^s_n(y, \mathrm{d}w)\,Q^{t-s}_{n+1}(w, E)\,\hat{\mu}(\mathrm{d}s).
$$
Assuming we proved this for $m = n + k$, let us prove this for $m = n + k + 1$. Similarly to~\eqref{eq:pve-k}, but integrating over $D_{n+k}\cap D_{n+k+1}$, we get:
$$
P^t((n, z), \{m\}\times E) \ge \int_{D_{n+k}}P^{t/2}(y, (n+k, \mathrm{d}w))\,P^{t/2}((n+k, w), \{n+k+1\}\times E) > 0.
$$
This completes the proof of the induction step, and with it the proof of~\eqref{eq:pvity-n} in case (a). 

\smallskip

{\it Case (b).} $z \in D_n,\, E\cap D_m = \varnothing$. (Clearly, we can reduce the case of a general $E$ to these two cases (a) and (b).) Since $E \subseteq D = \cup_k D_k$, there exists a $k$ such that $E \cap D_k$ has positive Lebesgue measure. Take the $k$ with such a property which is closest to $m$. The process can get from $(n, z)$ to $\{k\}\times (E\cap D_k)$ with positive probability in time $t/2$, using the path described in case (a) above. Afterwards, for every $z \in E\cap D_k$, the process $(N, Z)$ can jump from $(k, z)$ to $(m, z)$ in time $t/2$ with positive probability. Indeed, for $l$ between $k$ and $m$ we have $z \notin D_l$; thus the component $N$ will jump from $k$ to $m$, and the environment component $Z$ will stay constant at $z$.

\smallskip

{\it Case (c).} $z \notin D_n$. There exists a $k$ such that $z \in D_k$, since $z \in D = \cup_kD_k$. Find such $k$ wich is closest to $n$. The process $(N, Z)$ can get from $(n, z)$ to $(k, z)$ in time $t/2$ with positive probability:  The queue component $N$ will jump from $n$ to $k$, and the environment component $Z$ will stay constant at $z$, since $z \notin D_l$ for $l$ between $n$ and $k$. Starting the process from $(k, z)$ instead of $(n, z)$ now, we are back to cases (a) and (b). Applying results from these cases for $t/2$ instead of $t$, we prove~\eqref{eq:pvity-n} for $z \notin D_n$. 
\end{proof}

We proved Lemma~\ref{lemma:pve-adjusted}, and with it we proved ergodicity~\eqref{eq:long-term-conv}, and thus Theorem~\ref{thm:2.2}. 
\end{proof}

\medskip

Now we provide examples in which the generator of the diffusive component in the joint process depends on the queueing state in a nontrivial manner.

\begin{exm} \label{ex:2.1} Assume $D = [0, 1]$, $D_n = [0, \alpha_n]$, $\lambda(z) = z$, $\mu(z) = 1$. Assume ${\mathcal A}_n$ is a reflected diffusion (without jumps):
\begin{equation}
\label{eq:particular-generator}
{\mathcal A}_nf(z) = \frac{a^2(z)}{2}f''(z)+b(z)f'(z),\quad z \in D_n.
\end{equation}
The functions $a,b\in C^2_b([0,1])$ are given, describing the local diffusion 
coefficient and the local drift of the processes ${\widetilde Z}_n$ in $D_n$, with $a(z)>0$ for $z \in (0, 1)$. Standard formulas from \cite{Handbook} guarantee that the measure $\upsilon$ on $D$ has Lebesgue density 
\begin{equation}
\label{eq:invariant-upsilon-1}
q(z) = \frac{2}{a^2(z)}\exp\left(\int_0^z\frac{2b(y)}{a^2(y)}\,{\rm d}y\right)
\end{equation}
assuming that 
\begin{equation}
\label{eq:finite-measure-1}
\int_0^1\frac{|b(y)|}{a^2(y)}\,{\rm d}y < \infty.
\end{equation}
Assumption~\ref{asmp:finite-new-version} becomes
\begin{equation}
\label{eq:new-condition-1}
\sum\limits_{n=0}^{\infty}\int_0^{\alpha_n}\rho^n(z)q(z)\,\mathrm{d}z < \infty.
\end{equation}
In particular, for $a(z)\equiv 1$ and $b(z)=\theta/(z-1)$ with $\theta>0$, we get $q(z)=2(1-z)^{-2\theta}$. If we choose 
\begin{equation}
\label{alphan}
\alpha_n = 
\begin{cases}
1,\quad n < n_0;\\
\alpha_*,\quad n \ge n_0,
\end{cases}
\end{equation}
for some $\alpha_*\in (0,1)$ and $n_0\in \NN$, then~\eqref{eq:new-condition-1} holds for $\theta \in (0, 1/2)$. If $a \equiv 1$ and $b \equiv 0$, then the driving process for the environment is a reflected Brownian motion, with $\upsilon$ being the Lebesgue measure, and $q(z) \equiv 1$. 

This example can be interpreted as follows. We keep the service rate fixed: $\mu =1$, while the arrival rate $\lambda$ varies as a reflected diffusion on $[0,1]$ if the queue size $n$ is less than an agreed threshold $n_0$. However, if $n$ reaches level $n_0$ while $\lambda< \alpha^*$, we allow $\lambda$ to vary only in a ``safety range'' $[0, \alpha^*]$. If $n$ attains level $n_0$ while $\lambda \ge \alpha^*$,  we simply ``freeze" $\lambda$ until the queue size becomes $n_0-1$, at which time $\lambda$ is again allowed to follow the diffusion on $[0,1]$.  
\end{exm}

\smallskip

\begin{exm}\label{ex:2.2} Fix the arrival rate $\lambda =1$ while the service rate $\mu_n$ is subject to a reflected diffusion on the interval $D_n:=[\alpha_n,\alpha^*]\subset [1, \alpha^*]=:D$ 
and kept unchanged in $D\setminus D_n$. Here $\alpha_*>1$ is a fixed constant. 
Here again, the generator ${\mathcal A}_n$ is given by~\eqref{eq:particular-generator}, with $a,b\in{C}_b^2([1,\alpha^*])$; this operator from~\eqref{eq:particular-generator} acts on $f\in {C}^2([\alpha_n,\alpha^*])$ with boundary conditions $f'(\alpha_n)=f'(\alpha^*) = 0$. Instead of~\eqref{eq:invariant-upsilon-1}, we have
\begin{equation}
\label{eq:invariant-upsilon-2}
q(z) = \frac{2}{a^2(z)}\exp\left(\int_{1}^z\frac{2b(y)}{a^2(y)}\,{\rm d}y\right).
\end{equation}
and instead of assumption~\eqref{eq:finite-measure-1}, we have
\begin{equation}
\label{eq:finite-measure-2}
\int_1^{\alpha^*}\frac{|b(y)|}{a^2(y)}\,{\rm d}y < \infty.
\end{equation}
Assumption~\ref{asmp:finite-new-version} becomes
\begin{equation}
\label{eq:new-condition-2}
\sum\limits_{n=0}^{\infty}\int_{\alpha_n}^{\alpha^*}\rho^n(z)q(z)\,\mathrm{d}z < \infty.
\end{equation}
As in Example~\ref{ex:2.1}, if $a \equiv 1$, $b \equiv 0$, then the driving process for the environment is a reflected Brownian motion, with $\upsilon$ being the Lebesgue measure, and $q(z) \equiv 1$. 
\end{exm}

\medskip

\section{Explicit rates of exponential convergence}
\label{sec-rate}

\subsection{A brief summary of results and methods} In this section, we prove (for both discrete-space and reflected diffusion environments) that for some constants $C, \varkappa > 0$, we have
\begin{equation}\label{eq:exp-ergodicity}
\norm{P^t(x, \cdot) - \pi(\cdot)}_{\TV} \le C(x)e^{-\varkappa t},\, x \in D,\, t \ge 0,
\end{equation}
and estimate the constant $\varkappa$. We do this by {\it coupling:} Take two copies $(N_1, Z_1)$ and $(N_2, Z_2)$ of this process starting from $x_1 = (n_1, z_1)$ and $x_2 = (n_2, z_2)$. Couple them (that is, construct them on the same probability space) such that the {\it coupling time} 
$$
\tau := \inf\{t \ge 0\mid N_1(t) = N_2(t),\, Z_1(t) = Z_2(t)\}
$$
satisfies $\mathbb E\left[ e^{\varkappa \tau}\right] < \infty$ for some constant $\varkappa > 0$. By the standard Lindvall inequality we get
\begin{equation}
\label{eq:coupled}
\norm{P^t(x_1, \cdot) - P^t(x_2, \cdot)}_{\TV} \le \mathbb E\left[ e^{\varkappa \tau}\right] e^{-\varkappa t},\, x \in D,\, t \ge 0.
\end{equation}
We need only to integrate~\eqref{eq:coupled} with respect to $x_2 \sim \pi$ to get~\eqref{eq:exp-ergodicity}. To obtain such a coupling, we apply the following method. We wait until the queue component hits $0$ for both copies. Thus these queue components become coupled, that is,  they are at the same point. Then we wait until: (a) either one of these queue components jumps back to $1$, or  (b) the environment components become coupled. In case of (b), we have coupled both copies. In case of (a), we have failed, and need to repeat this procedure. Each time, we succeed with positive probability (bounded from below). Thus the number of tries is dominated by a geometric distribution. 

To couple the environment components, we use the results of \cite{MyOwn12}; however, it is well-known how to find hitting time of zero by the $M/M/1$ queue \cite{Robert-book}.  
Note that assuming exponential rates of convergence of $\mathcal{A}_n$ given each queue state $n$ does not immediately imply the exponential rate of convergence of the joint process $(N,Z)$. The particular multiplicative structure we consider in $\mathcal{A}_n$ enables us to obtain exponential estimates for the coupling time constructed for the joint processes $(N,Z)$ 
 under the mild conditions imposed on $\mathcal{A}$ as well as the arrival and service rates. 

\subsection{Main statements} We impose two assumptions. The first assumes exponential bounds  on the coupling time (uniform in state variables) associated with the generator $\mathcal{A}$.

\begin{asmp} The domain $D \subseteq \mathbb R^d$ is bounded. There exist constants $\alpha > 1$ and $\gamma > 0$ such that for all $z_1, z_2 \in D$ we can couple two processes $Z_1, Z_2$ with generator $\mathcal A$, starting from $Z_1(0) = z_1$ and $Z_2(0) = z_2$, in time $\tau_{z_1, z_2} := \inf\{t \ge 0\mid Z_1(t) = Z_2(t)\}$, with 
\begin{equation}
\label{eq:coupling-bound}
\mathbb P(\tau_{z_1, z_2} \ge t) \le \alpha e^{-\gamma t}.
\end{equation}
\label{asmp:basic}
\end{asmp}

\noindent The other assumption is a stronger condition on the traffic intensity: In previous sections, we assumed it is less than $1$, but now it has to be uniformly bounded away from $1$. 

\begin{asmp} There exist constants $\overline{\lambda}, \overline{\mu} > 0$ which satisfy 
$$
\lambda(z) \le \overline{\lambda} < \overline{\mu} \le \mu(z),\quad z \in D.
$$
\label{asmp:q}
\end{asmp}
\noindent From this Assumption~\ref{asmp:q},  
\begin{equation}
\label{eq:rho-bound}
\rho(z) \le \overline{\rho} := \frac{\overline{\lambda}}{\overline{\mu}} < 1,\quad z \in D. 
\end{equation}
Next, define the function
\begin{equation}
\label{eq:m-c}
m(c) := -\overline{\lambda}c - \overline{\mu}c^{-1} + (\overline{\lambda} + \overline{\mu}),\quad c \ge 1.
\end{equation}
This function is concave, increasing on $[1, c^*]$ and decreasing on $[c^*,\infty)$, with $c^*:=\overline{\rho}^{-1/2}$, and
$$
m(1) = 0,\quad m(c^*) =  \left(\sqrt{\overline{\mu}} - \sqrt{\overline{\lambda}}\right)^{2}.
$$
Finally, define the function 
\begin{equation}
\label{eq:theta}
\theta(\alpha, \beta, \gamma, a) := \frac{a\gamma}{(a-\beta)(\beta+\gamma-a)} \alpha^{(a - \beta)/\gamma} +  \frac{\beta}{\beta - a}.
\end{equation}
for any $\alpha>1$, $\beta, \gamma>0$ and $a \ge 0$.

\begin{thm} \label{thm:main-conv}
Fix an initial condition $x_0 = (n_0, z_0) \in \mathbb Z_+\times \mathbb R_+$. Under Assumptions~\ref{asmp:basic} and~\ref{asmp:q}, for some constants $C  > 0$ and $c \in (1, c_*)$,
\begin{equation} \label{eq:explicit-pi}
\norm{P^t((n_0, z_0), \cdot) - \pi(\cdot)}_{\TV} \le C\left(1 + c^{n_0}\right)e^{-\varkappa t},\,\, t \ge 0.
\end{equation}
where we can take any $\varkappa = (1 - \varepsilon)m(c)$ for $\varepsilon \in (0, 1)$ and $c \in (1, c^*)$ such that 
\begin{equation}
\label{eq:eps-c}
c\theta(\alpha, \overline{\lambda}, \gamma, m(c)) < \Bigl(1 - \alpha^{-\overline{\lambda}/\gamma}\frac{\gamma}{\overline{\lambda} + \gamma}\Bigr)^{-\varepsilon/(1 - \varepsilon)}.
\end{equation}
\end{thm}

The proof of the theorem is given at the end of this section. The only condition on the environment process is the Assumption~\ref{asmp:basic} on coupling time with (uniformly) exponential tail for the environment process corresponding to $N(t) = 0$. It is natural to assume this condition also holds for finite environment space. 

\smallskip

Note that there exists a $c \in (1, c^*)$ such that \eqref{eq:eps-c} is satisfied. Indeed, the left-hand side of \eqref{eq:eps-c} is continuous with respect to $c$, and is equal to $1$ for $c = 1$. Whereas the right-hand side of \eqref{eq:eps-c} is larger than one for any $\epsilon \in (0,1)$. However, to find a maximal rate of convergence, one needs to maximize $\varkappa$ over the space of two parameters $(\varepsilon, c)$ which satisfy~\eqref{eq:eps-c}. Possible values of $\varkappa$ form an interval $[0, \varkappa_*)$, which does not contain its upper endpoint; therefore, we cannot claim that $\varkappa_*$ is itself a rate of convergence. 

\smallskip

Compare this with the simple $M/M/1$ queue with constant rates: arrival rate $\overline{\lambda}$ and service rate $\overline{\mu}$, which has an exact rate of convergence 
$e^{-m(c)t}$ for $c > 0$ such that $m(c) > 0$ from~\eqref{eq:m-c}. \cite[Proposition 5.8]{Robert-book} states that the upper bound, restricting to only the queueing process, is
$$
\big(1+ \overline{\rho}^{-n/2} \big) \exp\big[ - \big( \overline{\lambda}^{1/2} - \overline{\mu}^{1/2} \big)^2 t \big]
$$
The constant in the exponent does not depend on $n$. Our result matches this rate.

After some modifications, this theorem is applicable not only for reflected diffusions from Section 2, but for discrete environment space from Section 3. Here is its version:

\begin{asmp}
There exist constants $\alpha > 1$ and $\gamma > 0$ such that for all $z_1, z_2 \in D$ we can couple two continuous-time Markov chains $Z_1, Z_2$ with common generator $\sigma(\cdot)\mathbf{T}_0$, starting from $Z_1(0) = z_1$ and $Z_2(0) = z_2$, in time $\tau_{z_1, z_2}$, such that~\eqref{eq:coupling-bound} holds.
\label{asmp:basic-new}
\end{asmp}

\begin{thm}
Under Assumptions~\ref{asmp:q} and \ref{asmp:basic-new}, the result~\eqref{eq:explicit-pi} for $(c, \varepsilon)$ satisfying~\eqref{eq:eps-c} holds. 
\label{thm:new-conv}
\end{thm}

\subsection{On the Assumptions~\ref{asmp:basic} or~\ref{asmp:basic-new}}
Below we give examples of discrete and continuous environment processes which satisfy Assumptions~\ref{asmp:basic} or~\ref{asmp:basic-new}.

\subsubsection{Coupling of jump processes}
First, let us start with discrete-space Markov chains. The relation between coupling times and mixing times (for $P^t(x, \cdot)$ to converge within a fixed $\TV$ distance from the stationary distribution) is partially explored in \cite{Finite}.  
There is a lot of existing literature on mixing times. For example, an extensive treatment of mixing times is given by \cite{PerezBook}. The literature on coupling times is sparse. Much of  the existing research is focused on $\gamma$ from Assumption~\ref{asmp:basic}, see for example \cite{BurdzyKendall}, but we need to know both $\alpha$ and $\gamma$. We could not find articles which estimate both of them. Thus we present an elementary result, which we hope will be useful. The proof is in the Appendix.

\begin{lemma} \label{lem-coupling-ex1}
Take a pure jump Markov process on the state space $D$ (finite, countable, or a domain in $\mathbb R^d$) such that the family of jump measures $(\nu(x, \cdot))_{x \in D}$ obeys 
$$
\Lambda := \sup\limits_{x \in D}\lambda(x),\quad \lambda(x) := \nu(x, D),\quad x \in D,
$$
and the family of probability measures 
$$
\overline{\nu}(x, \cdot) := \frac1{\Lambda}\nu(x, \cdot) + \frac{\Lambda - \lambda(x)}{\Lambda}\delta_{\{x\}},\quad x \in D,
$$
satisfies the following condition:
\begin{equation}
\label{eq:condition-q}
q := \sup\limits_{x, y \in D}\norm{\overline{\nu}(x, \cdot) - \overline{\nu}(y, \cdot)}_{\TV} < 1.
\end{equation}
Then the coupling times $\tau_{x, y}$ satisfy the following uniform estimate:
$$
\mathbb P(\tau_{x, y} \ge t) \le \exp\big(-(1 - q)\Lambda t\big).
$$
\label{lemma:pure-jump}
\end{lemma}

\begin{rmk}
The same result is true if the process is a reflected jump-diffusion with jump measures satisfying conditions of Lemma~\ref{lemma:pure-jump}. 
\end{rmk}

\begin{exm} The condition~\eqref{eq:condition-q} is not true if at least two measures $\overline{\nu}(x, \cdot)$ and $\overline{\nu}(y, \cdot)$ are mutually singular; that is, there exists a set $D_0 \subseteq D$ such that $\overline{\nu}(x, D_0) = 0$ but 
$\overline{\nu}(y, D_0) = 1$. Indeed, we then have 
$$
\norm{\nu(x, \cdot) - \nu(y, \cdot)}_{\TV} \ge |\nu(x, D_0) - \nu(y, D_0)| = 1.
$$
\end{exm}

\begin{exm}
Assume that for all $x \in D$, $\nu(x, \cdot) \ll \mu(\cdot)$ for some $\sigma$-finite Borel measure $\mu$ on $D$. It can be the Lebesgue measure if $D$ is a domain in $\mathbb R^d$, or the counting measure for discrete $D$. Define the Radon-Nikodym derivative
$$
f(x, z) := \frac{\mathrm{d}\nu(x, \cdot)}{\mathrm{d}\mu(\cdot)}(z).
$$
Then condition~\eqref{eq:condition-q} is equivalent to 
$$
\sup\limits_{x, y \in D}\int_D|f(x, z) -  f(y, z)|\mathrm{d}\mu(z) = q < 1.
$$
For example, take a finite $D$ (with $m$ elements). Let $\mu$ be the counting measure, then $\nu(\cdot, \cdot)$ can be given by an $m\times m$ matrix $(\nu_{ij})$ (with zero diagonal elements). Each $i^{\rm th}$ row gives Radon-Nikodym derivative of $\nu(i, \cdot)$ with respect to $\mu$. Thus we obtain
$$
q := \max\limits_{i, j = 1, \ldots, m}\sum_{k=1}^m|\nu_{ik} - \nu_{jk}|.
$$
\end{exm}

\subsubsection{Coupling of reflected diffusions} Now consider a reflected diffusion on $[0, a]$. It is stochastically ordered, so every $\tau_{x, y}$ is stochastically dominated by $\mathcal T$: hitting time of $a$ starting from $0$. Thus
$$
\mathbb P(\tau_{x, y} \ge t) \le \mathbb P(\mathcal T \ge t).
$$
Let us estimate the tail of $\mathcal T$. Take a non-reflected diffusion $Z^*= \{Z^{*}(t):\, t \ge 0\}$ on the real line, with drift and diffusion coefficients
$$
g^*(x) = 
\begin{cases}
g(x),\, x \ge 0,\\
-g(-x),\, x < 0,
\end{cases}
\quad
\sigma^*(x) = \sigma(|x|),\quad x \in \mathbb R.
$$
Let $\mathcal T^* := \inf\{t \ge 0: |Z^*(t)| = a\}$. Then the laws of $Z(\cdot\wedge \mathcal T)$ and $Z^*(\cdot\wedge\mathcal T^*)$ are the same, and the laws of $\mathcal T$ and $\mathcal T^*$ are the same. Thus we have  reduced this to tail estimation for an exit time of a diffusion process from a strip $[-a, a]$. 

Denote by $u^*(t, x)$ the probability that $Z^*$ stays in $(-a, a)$ until at least time $t$, if $Z^*(0) = x$. Denote by $G(t, x, y)$ the transition density of this diffusion killed at $\pm a$, otherwise known as Green's function (or heat kernel) of the infinitesimal generator $\mathcal A^*$ of $Z^*$. Then the function $u^*$ satisfies the initial-boundary value problem
$$
\frac{\partial u^*}{\partial t} = \mathcal A^*u^*,\quad t \ge 0,\quad -a < x < a,
$$
with initial and boundary conditions $\left.u^*\right|_{t=0} = 1$ and 
$\left.u\right|_{x = \pm a} = 0$. Thus we can express 
$$
u^*(t, x) = \int_{-a}^aG(t, x, y)\,\mathrm{d}y.
$$
Knowing spectral decomposition of $G$ gives us the exponent in \eqref{eq:coupling-bound}. To find the constant $A$ is a little harder, since it requires some information on the function $G$ itself, or its eigenvalues. In some simple cases, however, it can be found explicitly. For example, for a RBM $Z$ on $[0, a]$, the process $Z^*$ is also a Brownian motion, and \cite[Chapter 2, Problem 8.2]{KSBook} gives us an exact estimate. 

\subsection{Proof of Theorem~\ref{thm:new-conv}} We proceed in seven steps. 

\smallskip

 {\it Step 1.} It suffices to prove the following version of~\eqref{eq:coupled}: For $(n_1, z_1), (n_2, z_2) \in \mathbb Z_+\times D$,
\begin{equation}
\label{eq:explicit}
\norm{P^t((n_1, z_1), \cdot)  - P^t((n_2, z_2), \cdot)}_{\TV} \le C_*\left(c^{n_1} + c^{n_2}\right)e^{-\varkappa t},\quad t \ge 0,
\end{equation}
for some constant $C_*$ (which will be determined below). 
Indeed, then we can rewrite~\eqref{eq:explicit} as follows: For every Borel subset $A \subseteq \mathbb Z_+\times D$, 
\begin{equation}
\label{eq:explicit-version}
\left|P^t((n_1, z_1), A) - P^t((n_2, z_2), A)\right| \le C_*\left(c^{n_1} + c^{n_2}\right)e^{-\varkappa t}.
\end{equation}
Integrate~\eqref{eq:explicit-version} with respect to $(n_2, z_2) \sim \pi$. Note that the function $(n, z) \mapsto c^n$ is integrable with respect to $\pi$. Indeed, this integral is equal to 
$$
\Xi^{-1}\sum\limits_{n=0}^{\infty}\int_Dc^n\rho^n(z)\nu_D(\mathrm{d}z).
$$
From~\eqref{eq:rho-bound}, $\nu_D(D) = 1$, and $c < c_* = \overline{\rho}^{-1/2}$, 
\begin{equation}
\label{eq:integrate}
\sum\limits_{n=0}^{\infty}\int_Dc^n\rho^n(z)\nu_D(\mathrm{d}z) \le \sum\limits_{n=0}^{\infty}\overline{\rho}^{n/2} = (1 - \overline{\rho}^{1/2})^{-1} < \infty.
\end{equation}
Combining~\eqref{eq:explicit-version} and~\eqref{eq:integrate}, we get~\eqref{eq:explicit-pi}.

\smallskip

{\it Step 2.} To get~\eqref{eq:explicit}, we use {\it coupling:} As explained in the beginning of this section, we take on the same filtered probability space two copies $X_1 = (N_1, Z_1)$ and $X_2 = (N_2, Z_2)$ of this queue, starting from $x_1 = (n_1, z_1)$ and $x_2 = (n_2, z_2)$. Assume $\tau\equiv \tau(x_1,x_2)$ is a stopping time such that $X_1(t) = X_2(t)$ for $t \ge \tau$ a.s. Then $\tau$ is called a {\it coupling time}. For every $t \ge 0$ and a function $f : \mathbb Z_+ \times D \to \mathbb R$ with $|f| \le 1$, we can write
\begin{align}
\label{eq:coupling-technique}
\begin{split}
&\left|\mathbb E f(X_1(t)) - \mathbb E f(X_2(t))\right|  \le 
\left|\mathbb E\left[f(X_1(t))1_{\{\tau \le t\}}\right] - \mathbb E\left[f(X_2(t))1_{\{\tau \le t\}}\right]\right| \\ & \qquad \qquad\qquad\qquad\qquad\qquad + \left|\mathbb E\left[f(X_1(t))1_{\{\tau > t\}}\right] - \mathbb E\left[f(X_2(t))1_{\{\tau > t\}}\right]\right|  \le 2\mathbb P(\tau > t).
\end{split}
\end{align}
In other words, we get the classic Lindvall inequality
\begin{equation}
\label{eq:Lindvall}
\left|\mathbb E f(X_1(t)) - \mathbb E f(X_2(t))\right|  \le 2\mathbb P(\tau > t).
\end{equation}
Next, assuming that we prove that 
$\mathbb E e^{\varkappa\tau} < \infty$, then 
\begin{equation}
\label{eq:Markov}
\mathbb P(\tau > t) \le e^{-\varkappa t}\cdot\mathbb E e^{\varkappa\tau}. 
\end{equation}
Combining~\eqref{eq:Lindvall} with~\eqref{eq:Markov}, we get~\eqref{eq:explicit-version}. In the proof below, we shall see that the constant before $e^{-\varkappa t}$ turns out to be of the same form as required in~\eqref{eq:explicit-version}.

\smallskip

{\it Step 3.} Let us now describe the coupling in detail. 

\smallskip

(a) First, we couple the queue components. Both $N_1$ and $N_2$ are stochastically dominated by $\overline{N}$, which is described as the $M/M/1$ queue with arrival rate $\overline{\lambda}$ and service rate $\overline{\mu}$, starting from $\overline{N}(0) =  n_1\vee n_2$. Therefore, we can take copies of $N_1, N_2, \overline{N}$ such that 
\begin{equation}
\label{eq:stoch-domination-N}
N_1(t) \le \overline{N}(t)\ \mbox{and}\ N_2(t) \le \overline{N}(t), \quad t \ge 0. 
\end{equation}
From~\eqref{eq:stoch-domination-N} it follows that for $\tau_0 := \inf\{t\ge 0\mid \overline{N}(t) = 0\}$, we have $N_1(\tau_0) = N_2(\tau_0) = 0$. 

\smallskip

(b) At $\tau_0$, we start two competing clocks. The first one is an exponential clock $\eta_0 \sim \Exp(\overline{\lambda})$, which measures the time until arrival of the process $\overline{N}$ to $1$ from $0$. The second one is $\zeta_0$, a coupling time of $Z_1(\tau_0+\cdot)$ and $Z_2(\tau_0 + \cdot)$. This time $\zeta_0$ exists by Assumption~\ref{asmp:basic}, since these two processes are copies of the environment process with generator $\mathcal A$ (recall $\beta_0 = 1$) starting from $Z_1(\tau_0)$ and $Z_2(\tau_0)$, respectively.  At least (importantly for us here), this is true until $\eta_0$, when those drift and diffusion coefficients change. 

\smallskip

(c) If $\zeta_0 < \eta_0$, then $Z_1$ and $Z_2$ have time to couple while $\overline{N}(t) = 0$. By stochastic domination, $N_1(t) = N_2(t) = 0$. Thus $S_0 := \tau_0 + \zeta_0$ is a coupling time for $X_1$ and $X_2$. 

\smallskip

(d) If, however, $\zeta_0 \ge \eta_0$, then the coupling did not work. The process $\overline{N}$ has jumped at time $\tau_0 + \eta_0$ back to $1$, and we need to repeat this procedure. Let 
$$
\tau_1 := \inf\{t \ge 0\mid \overline{N}(t + \tau_0 + \eta_0) = 0\},\quad \eta_1 \sim \Exp(\overline{\lambda}).
$$
Let $\zeta_1$ be a coupling time of $Z_1(\tau_1 + \tau_0 + \eta_0 + \cdot)$ and $Z_2(\tau_1 + \tau_0 + \eta_0 + \cdot)$. If $\zeta_1 < \eta_1$, then for $S_1 := \tau_0 + \eta_0 + \tau_1 + \zeta_1$ we have $\overline{N}(S_1) = 0$, 
and thus $N_1(S_1) = N_2(S_1) = 0$. But since $\zeta_1$ is also a coupling time for environment components, 
$Z_1(S_1) = Z_2(S_1)$. Thus $S_1$ is a coupling time for $(N_1, Z_1)$ and $(N_2, Z_2)$.  

\smallskip 

(e) If $\zeta_1 \ge \eta_1$, then this coupling did not work, and we need to repeat this procedure, with $\zeta_2, \eta_2, S_2$, and so on. Let $\mathcal J := \min\{j \ge 0\mid \zeta_j < \eta_j\}$. Then the ultimate coupling time is
\begin{equation}
\label{eq:ultimate-coupling}
\tau := \sum_{j=0}^{\mathcal J-1}(\tau_j + \eta_j) + \tau_{\mathcal J} + \zeta_{\mathcal J} = \sum\nolimits_{j=0}^{\mathcal J}(\tau_j + \eta_j\wedge\zeta_j) = S_{\mathcal J},
\end{equation}
where we define the following random times:
\begin{equation}
\label{eq:ultimate-notation}
S_k := \sum_{j=0}^k\xi_j,\quad \xi_k := \tau_k + \zeta_k\wedge\eta_k,\quad k \in \mathbb Z_+.
\end{equation}
Next, we estimate the MGF of $\tau$ from~\eqref{eq:ultimate-coupling}. 

\smallskip

{\it Step 4.}  First, we estimate the MGF for each $\tau_k$. The generator of $\overline{N}$ is
$$
\overline{\mathcal M}f(n) = \overline{\lambda}(f(n+1) - f(n)) + \overline{\mu}1_{\{n \ne 0\}}(f(n-1) - f(n)).
$$
Therefore, letting $f(n) = c^n$ for a constant $c > 1$, we get
$$
\overline{\mathcal M}f(n) = -m(c)f(n),\,  \quad n \ge 1,
$$
with the constant $m(c)$ defined in \eqref{eq:m-c}. 
The following process
$$
L(t) := c^{\overline{N}(t\wedge\tau_0)} + m(c)\int_0^{t\wedge\tau_0}c^{\overline{N}(s)}\,\mathrm{d}s,\, \quad t \ge 0,
$$
is a local supermartingale, because the function $W_N : n \mapsto c^n$ satisfies 
$$
\overline{\mathcal M}W_N(n) \le -m(c)W_N(n),\,\quad  n = 1, 2, \ldots. 
$$
In the terminology of \cite[Section 4]{MyOwn12}, this is a {\it modified Lyapunov function} for $\overline{N}$. Then the derivation is similar to \cite[Section 5]{MyOwn12}. By Fatou's lemma, $L$ is a true supermartingale. Let
$$
L_*(t) := \int_0^te^{m(c)s}\,\mathrm{d}L(s),\, 
\quad t \ge 0.
$$
Because  $e^{ms} \ge 0$, this process is also a supermartingale. Consider the process
$$
L^*(t) := e^{m(c)(t\wedge\tau_0)}c^{\overline{N}(t\wedge\tau_0)}, \quad t\ge 0. 
$$
By an elementary calculation, $\mathrm{d}L^*(t) = \mathrm{d}L_*(t)$. Therefore $L^*(t) = L_*(t) + \mathrm{const}$, and $L^*$ is itself a supermartingale. Thus, for every $t \ge 0$, 
\begin{equation}
\label{eq:supermart}
\mathbb E\left[e^{m(c)(t\wedge\tau_0)}c^{\overline{N}(t\wedge\tau_0)}\right] \le \mathbb E c^{\overline{N}(0)}.
\end{equation}
Let $t \to \infty$ in~\eqref{eq:supermart}. By Fatou's lemma with the observation that $\overline{N}(\tau_0) = 0$, we get
\begin{equation}
\label{eq:MGF-tau}
\mathbb E e^{m(c)\tau_0} \le c^{n_1\vee n_2}.
\end{equation}
Similarly to~\eqref{eq:MGF-tau}, we get estimates for the MGFs  of $\tau_1, \tau_2, \ldots$, with the difference that the initial state becomes $1$ instead of $n_1\vee n_2$. Therefore, 
\begin{equation}
\label{eq:mgf-tau}
\mathbb E e^{m(c)\tau_k} \le c,\quad k = 1, 2, \ldots. 
\end{equation}

\smallskip

{\it Step 5.} 
By Assumption~\ref{asmp:basic}, we have $P(\zeta_k > t) \le \alpha e^{-\gamma t}$ for $t>0$, and recall that $\eta_k \sim \Exp(\bar{\lambda})$. Also, $\zeta_k$ and $\eta_k$ are independent.  Thus, by Lemma~\ref{lemma:tech}, we have
 for all $k \in \mathbb Z_+$, 
$$
\mathbb P(\zeta_k \le \eta_k) \le  \frac{\gamma}{\overline{\lambda}+\gamma}\alpha^{-\overline{\lambda}/\gamma} =: p.
$$
 Thus  the number of `tries', $\mathcal J$, is stochastically dominated by a geometric random variable $\widetilde{\mathcal J}$, which is the number of trials that one needs to get to the first success if the probability of success of each trial is $p$. It has the distribution and generating function (with $q := 1-p$)
\begin{equation}
\label{eq:Geo}
\mathbb P(\widetilde{\mathcal J}= n) = pq^{n-1},\, n = 1, 2, \ldots, \qandq  \mathbb E\left[s^{\widetilde{\mathcal J}}\right] = \frac{ps}{1 - qs},\ s \in [0, q^{-1}).
\end{equation}

\smallskip

{\it Step 6.} Let us estimate the MGF of $\xi_k$, defined in \eqref{eq:ultimate-notation}. 
 By Assumption~\ref{asmp:basic} and Lemma~\ref{lemma:tech} applied to $a := m(c)$ for $c \in [1, c_*]$, 
\begin{equation}
\label{eq:mgf-min}
\mathbb E\left[e^{m(c)(\zeta_k\wedge\eta_k)}\right] \le \theta(\alpha, \overline{\lambda}, \gamma, m(c)).
\end{equation}
The expression for $\theta(\alpha, \beta, \gamma, a)$ is given in~\eqref{eq:theta}. Combining~\eqref{eq:mgf-tau} and~\eqref{eq:mgf-min}, we get
$$
\mathbb E\left[e^{m(c)\xi_k}\right] \le c\theta(\alpha, \overline{\lambda}, \gamma, m(c)) =: \kappa(c),\quad k = 1, 2, \ldots
$$
The same holds if we do conditional expectation
\begin{equation}
\label{eq:one-term}
\mathbb E\left[e^{m(c)\xi_k}\mid \mathcal F_{S_{k-1}}\right] \le c\theta(\alpha, \overline{\lambda}, \gamma, m(c)) ,\quad k = 1, 2, \ldots
\end{equation}
Combining~\eqref{eq:MGF-tau} and~\eqref{eq:mgf-min}, we get
\begin{equation}
\label{eq:first-term}
\mathbb E\left[e^{m(c)\xi_k}\right] \le c^{n_1\vee n_2}\theta(\alpha, \overline{\lambda}, \gamma, m(c)).
\end{equation}

\smallskip

{\it Step 7.} Finally, recall~\eqref{eq:ultimate-coupling}.
We estimate from above the MGF for appropriate $\varkappa > 0$: $\mathbb E\left[e^{\varkappa S_{\mathcal J}}\right] = \mathbb E\left[e^{\varkappa \tau}\right]$. By~\eqref{eq:one-term}, the process $(M_k)_{k \in \mathbb Z_+}$ defined by 
$$
M_k := \exp\big(m(c)S_k - k\ln\kappa(c)\big),\, \quad k \in \ZZ_+,
$$
is an $(\mathcal F_{S_k})_{k \in \mathbb Z_+}$-supermartingale. It is positive, and $\mathcal J$ is an $(\mathcal F_{S_k})_{k \in \mathbb Z_+}$-stopping time. Applying the optional stopping theorem and using~\eqref{eq:first-term}, we obtain
\begin{equation}
\label{eq:comparison-of-expectations}
\mathbb E\left[M_{\mathcal J}\right] \le \mathbb E[M_0] = \mathbb E\left[e^{m(c)\xi_0}\right] = c^{n_1\vee n_2}\theta(\alpha, \overline{\lambda}, \gamma, m(c)).
\end{equation}

By H{\"o}lder's inequality, 
\begin{align}
\label{eq:Holder}
\begin{split}
\mathbb E&\left[\exp\left((1 - \varepsilon)m(c)S_{\mathcal J}\right)\right] \\ & \le \left(\mathbb E\left[e^{m(c)S_{\mathcal J} - \mathcal J\ln\kappa(c)}\right]\right)^{1 - \varepsilon}\cdot\bigl(\mathbb E\bigl[\exp\left(\mathcal J((1-\varepsilon)/\varepsilon)\ln\kappa(c)\right)\bigr]\bigr)^{\varepsilon} \\ & =  \left(\mathbb E\left[M_{\mathcal J}\right]\right)^{1 - \varepsilon}
\mathbb E\left[\kappa(c)^{(1-\varepsilon)\mathcal J/\varepsilon}\right].
\end{split}
\end{align}
Since $\kappa(c) > 0$ for $c \in [1, c_*]$, and $\mathcal J$ is stochastically dominated by a geometric random variable $\widetilde{\mathcal J}$ as in~\eqref{eq:Geo}, we have
\begin{equation}
\label{eq:domination-geo}
\mathbb E\left[\kappa(c)^{(1-\varepsilon)\mathcal J/\varepsilon}\right] \le \mathbb E\left[\kappa(c)^{(1 - \varepsilon)\widetilde{\mathcal J}/\varepsilon}\right] = \frac{p\kappa(c)^{(1-\varepsilon)/\varepsilon}}{1 - \kappa(c)^{(1-\varepsilon)/\varepsilon}q}.
\end{equation}
Here we require that $\kappa(c)^{(1-\varepsilon)/\varepsilon} < q^{-1}$, which is exactly the condition for $c$ in \eqref{eq:eps-c}. Combining~\eqref{eq:comparison-of-expectations}, ~\eqref{eq:Holder} and ~\eqref{eq:domination-geo}, we get 
\begin{align}
\label{eq:final-eq}
\begin{split}
\mathbb E\left[\exp\left((1 - \varepsilon)m(c)S_{\mathcal J}\right)\right] &
 \le c^{(1-\epsilon) (n_1\vee n_2)}\theta(\alpha, \overline{\lambda}, \gamma, m(c))^{1-\epsilon}  \frac{p\kappa(c)^{(1-\varepsilon)/\varepsilon}}{1 - \kappa(c)^{(1-\varepsilon)/\varepsilon}q} \nonumber \\
& = C_* c^{(1-\epsilon) [(n_1\vee n_2) -1]} < C_* c^{n_1\vee n_2}  \le C_*(c^{n_1} + c^{n_2}),\\
  C_* &:=\frac{p\kappa(c)^{(1-\varepsilon)(1/\varepsilon + 1)}}{1 - \kappa(c)^{(1-\varepsilon)/\varepsilon}q}. 
\end{split}
\end{align}
From~\eqref{eq:ultimate-coupling}, this completes the proof of~\eqref{eq:explicit} for $\varkappa := (1 - \varepsilon)m(c)$, and Theorem~\ref{thm:main-conv}.

\section{Appendix}
\label{sec-appendix}
\subsection{Proof of Lemma \ref{lem-coupling-ex1}}
Alternatively we can describe such pure jump process $X = (X(t),\, t \ge 0)$ as follows: Run an exponential clock $\eta_1 \sim \Exp(\Lambda)$, and then let $X(t) = X(0)$ for $t < \eta_1$, and $X(\eta_1) \sim \overline{\nu}(X(0), \cdot)$ (independently of $\eta_1$). Run another exponential clock $\eta_2 \sim \Exp(\Lambda)$ independent of those random variables, then $X(S_2), Y(S_2)$ with $S_2 := \eta_1 + \eta_2$, and repeat the process. Thus we couple these processes $X = \{X(t):\, t \ge 0\}$ and $Y = \{Y(t):\, t \ge 0\}$ starting from $X(0) = x$ and $Y(0) = y$ as follows: We use the same exponential clocks $\eta_1, \eta_2, \ldots$, and couple $X(S_k)$ and $Y(S_k)$ with $S_k := \eta_1 + \ldots + \eta_k$, using the {\it maximal coupling} from \cite[Proposition 4.7]{PerezBook}: 
\begin{equation}
\label{eq:max-coupling}
\mathbb P\left(X(\eta_k) \ne Y(\eta_k)\mid \mathcal F_{S_{k-1}}\right) = \norm{\overline{\nu}(X(\eta_{k-1}, \cdot) -  \overline{\nu}(Y(\eta_{k-1}, \cdot))}_{\TV},\quad k = 1, 2, \ldots
\end{equation}
The coupling time then becomes 
\begin{equation}
\label{eq:tau-representation}
\tau_{x, y} := S_{\mathcal J},\quad \mathcal{J} := \min\{k \ge 1:  X(\eta_k) = Y(\eta_k)\}.
\end{equation}
Combining~\eqref{eq:condition-q} and~\eqref{eq:max-coupling}, we get
\begin{equation}
\label{eq:less-than-q}
\mathbb P\left(X(\eta_k) = Y(\eta_k)\mid \mathcal F_{S_{k-1}}\right) \ge p := 1-q,\quad k = 1, 2, \ldots. 
\end{equation}
Therefore, $\mathcal J$ is stochastically dominated by a geometric random variable $\tilde{\mathcal J}$ (the number of tries until the first success in a sequence of independent Bernoulli trials with individual success probability $p$), independent of $\eta_1, \eta_2, \ldots$. From~\eqref{eq:tau-representation} we get
\begin{equation}
\label{eq:comparison}
\tau_{x, y} \preceq \eta_1 + \ldots + \eta_{\tilde{\mathcal J}} =: \tilde{S}. 
\end{equation}
The MGF of each of these exponential random variables is 
$$
\mathbb E\left[ e^{u\eta_k}\right] = \frac{\Lambda}{\Lambda - u},\quad u < \Lambda,
$$
and the generating function for this geometric random variable is 
$$
\mathbb E\big[s^{\tilde{\mathcal J}}\big] = \frac{ps}{1 - qs},\quad s < q^{-1}.
$$
Therefore, the MGF for $\tilde{S}$ from the right-hand side of~\eqref{eq:comparison} is the composition: 
$$
\mathbb E\big[e^{u\tilde{S}}\big] = \frac{p\frac{\Lambda}{\Lambda - u}}{1 - q\frac{\Lambda}{\Lambda - u}} = \frac{p\Lambda}{p\Lambda - u}.
$$
Thus $\tilde{S} \sim \Exp(p\Lambda)$, and it satisfies $\mathbb P(\tilde{S} \ge t) \le e^{-p\Lambda t}$. The rest is trivial.

\subsection{A technical comparison lemma} 

\begin{lemma} Fix constants $\alpha > 1$, $\beta, \gamma > 0$. Take two independent random variables $\xi \sim \Exp(\beta)$ and $\eta > 0$ which satisfies
$\mathbb P(\eta > u) \le \alpha e^{-\gamma u}$ for $u \ge 0$. Then 
\begin{equation}
\label{eq:probab}
\mathbb P(\eta < \xi) \ge \alpha^{-\beta/\gamma}\frac{\gamma}{\beta + \gamma}.
\end{equation}
For $a \in [0, \beta + \gamma)$, the moment generating function for $\xi\wedge\eta$ satisfies
\begin{equation}
\label{eq:MGF-est}
\mathbb E\left[e^{a(\xi\wedge\eta)}\right] \le \theta(\alpha, \beta, \gamma, a),
\end{equation}
where the function $\theta$ is defined in~\eqref{eq:theta}. 
\label{lemma:tech}
\end{lemma}

\begin{proof} Let us first show~\eqref{eq:probab}. We have  $\alpha e^{-\gamma u} < 1$ for $u > u_0 := \gamma^{-1}\ln(\alpha)$. Then we can rewrite our tail estimate for $\eta$ as follows:
$$
\mathbb P(\eta \ge u) \le 
\begin{cases}
\alpha e^{-\gamma u},\,& u \ge u_0,\\
1,\, & u< u_0. 
\end{cases}
$$
Therefore, we have 
\begin{align*}
\mathbb P(\xi \le \eta) & = \int_0^{\infty}\beta e^{-\beta u}\mathbb P(u \le \eta)\,\mathrm{d}u  \\ & \le \int_{u_0}^{\infty}\alpha\beta e^{-\beta u}e^{-\gamma u}\,\mathrm{d}u + \int_0^{u_0}\beta e^{-\beta u}\,\mathrm{d}u \\ & = \frac{\alpha\beta}{\beta + \gamma}e^{-(\beta + \gamma)u_0} + (1 - e^{-\beta u_0}) 
 = 1 - \frac{\gamma}{\beta + \gamma}\alpha^{-\beta/\gamma}. 
\end{align*}
From here~\eqref{eq:probab} immediately follows. Next, let us show~\eqref{eq:MGF-est}. For every $u \ge 0$, 
\begin{align*}
\mathbb E&\left[e^{a(u\wedge\eta)}\right] = e^{au}\,\mathbb P(\eta > u) + \int_0^ue^{av}\,\mathbb P(\eta \in \mathrm{d}v) \\ & \le  e^{au}\,\mathbb P(\eta > u)  - \int_0^ue^{av}\,\mathrm{d}\mathbb P(\eta > v) \\ & = e^{au}\,\mathbb P(\eta > u)  - \left.e^{av}\,\mathbb P(\eta > v)\right|_{v=0}^{v=u} + \int_0^u\mathbb P(\eta > v)\,\mathrm{d} e^{av} \\ & \le  1 + \int_0^u(\alpha e^{-\gamma v}\wedge 1)\,ae^{av}\,\mathrm{d}v.
\end{align*}
Calculate the integral in the right-hand side by splitting it into two integrals: from $0$ to $u_0$ (where $u_0$ is defined above), and from $u_0$ to $u$. If $u \in [0, u_0]$,  this integral is equal to 
$$
\int_0^u(\alpha e^{-\gamma v}\wedge 1)\,ae^{av}\,\mathrm{d}v = \int_0^u\,ae^{av}\,\mathrm{d}v = e^{au} - 1.
$$
If $u > u_0$, then this integral is equal to 
\begin{align*}
\int_0^u&(\alpha e^{-\gamma v}\wedge 1)\,ae^{av}\,\mathrm{d}v = \int_0^{u_0}ae^{av}\,\mathrm{d}v + \int_{u_0}^u\alpha ae^{(a -\gamma)v}\,\mathrm{d}v \\ & = e^{au_0} - 1 + \frac{\alpha a}{a - \gamma}\left[e^{(a - \gamma)u} - e^{(a - \gamma)u_0}\right]  = \alpha^{a/\gamma} + \frac{\alpha a}{a - \gamma}\left[e^{(a - \gamma)u} - \alpha^{(a - \gamma)/\gamma}\right]  - 1 \\ & = 
\frac{\gamma}{\gamma - a}\alpha^{a/\gamma} + \frac{\alpha a}{a - \gamma}e^{(a - \gamma)u} - 1.
\end{align*}
Combining all these computations, we get
\begin{equation}
\label{eq:fixed-xi}
\mathbb E\left[e^{a(u\wedge\eta)}\right]  \le 
\begin{cases}
\frac{\gamma}{\gamma - a}\alpha^{a/\gamma} + \frac{\alpha a}{a - \gamma}e^{(a - \gamma)u},\, u > u_0;\\
e^{au},\, u \in [0, u_0].
\end{cases}
\end{equation}
Now integrate~\eqref{eq:fixed-xi} with respect to the exponential distribution of $\xi$: 
$\beta e^{-\beta u}\,\mathrm{d}u$:
\begin{align*}
\mathbb E&\left[ e^{a(\xi\wedge\eta)}\right] \le \int_0^{u_0}e^{au}\,\beta e^{-\beta u}\,\mathrm{d} u + \int_{u_0}^{\infty}\left[\frac{\gamma}{\gamma - a}\alpha^{a/\gamma} + \frac{\alpha a}{a - \gamma}e^{(a - \gamma)u}\right]\,\beta e^{-\beta u}\,\mathrm{d} u \\ & = \frac{\beta}{a - \beta}\left[e^{(a - \beta)u_0} - 1\right] + \frac{\gamma}{\gamma - a}\alpha^{a/\gamma}e^{-\beta u_0} + \frac{\alpha\beta a}{(a - \gamma)(\gamma + \beta - a)}e^{-(\gamma+\beta-a)u_0} \\ & = 
\frac{a(\beta - \gamma)}{(a - \beta)(a - \gamma)}\alpha^{(a - \beta)/\gamma} - \frac{\beta}{a - \beta} + \frac{\beta }{(a - \gamma)(\beta + \gamma - a)}\alpha^{1-(\gamma+\beta-a)/\gamma} \\ & 
= \frac{a\gamma}{(a-\beta)(\beta+\gamma-a)} \alpha^{(a - \beta)/\gamma} +  \frac{\beta}{\beta - a}.
\end{align*}
This completes the proof. 
\end{proof}

\section{Concluding Remarks} \label{sec-conclusion}

We have found the explicit invariant measure for the joint interactive queueing and environment process, and estimated the exponential rate of convergence for the compact environment case. 
One interesting question would be to consider unbounded environment domains, but with environment process being exponentially ergodic. This will require much finer estimates, because Assumption~\ref{asmp:basic} will hold only with $\alpha$ dependent on $z_1$ and $z_2$. One way to find such coupling was developed in \cite{Genius, MyOwn16, LMT1996, MyOwn12} via Lyapunov functions. 
Subgeometric rates of convergence seem interesting. Some work was done in \cite{Douc, Sub} for general Markov processes and in \cite{APS19, AHPS19} for some SDEs arising from many-server queues; but to the best of our knowledge none for our setup.

\section*{Acknowledgments}
G. Pang was supported in part by NSF grants CMMI-1635410, and DMS/CMMI-1715875 and in part by an Army Research Office grant W911NF-17-1-0019. 
Y. Belopolskaya was supported in part by RSF 17-11-01136. Y. Suhov thanks Department of Mathematics at Pennsylvania State university for
hospitality and support.

\end{document}